\documentclass[12pt,a4paper]{article}
\usepackage{amsfonts}
\usepackage{amsmath}
\usepackage{amsthm}
\usepackage{amssymb}
\usepackage{mathrsfs}
\usepackage[numbers]{natbib}
\usepackage[fit]{truncate}
\usepackage{fullpage}
\usepackage{hyperref}
\usepackage{tikz}
\usepackage{microtype}
\usepackage{mathtools}
\usepackage{enumitem}
\mathtoolsset{showonlyrefs}

\usepackage{titlesec}
\titlelabel{\thetitle.\;} 

\theoremstyle{plain}
\newtheorem{theorem}{Theorem}[section]
\newtheorem{corollary}[theorem]{Corollary}
\newtheorem{lemma}[theorem]{Lemma}

\newtheorem{proposition}[theorem]{Proposition}

\newtheorem{condition}[theorem]{Condition}
\newtheorem{conjecture}[theorem]{Conjecture}

\theoremstyle{definition}

\newtheorem{remark}[theorem]{Remark}

\usepackage{bbm}

\newcommand{\as}[0]{a.s.}
\newcommand{\ie}[0]{i.e.}

\newcommand{\ra}[0]{ \rightarrow }
\newcommand{\lra}[0]{ \longrightarrow }

\DeclareMathOperator{\Esp}{E}
\DeclareMathOperator{\Prob}{P}

\DeclareMathOperator{\Qrob}{Q}
\DeclareMathOperator{\II}{\mathbb{I}}
\DeclareMathOperator{\IR}{\mathbb{R}}
\DeclareMathOperator{\IN}{\mathbb{N}}

\DeclareMathOperator{\bF}{\mathcal{F}}
\DeclareMathOperator{\bP}{\mathcal{P}}
\DeclareMathOperator{\bE}{\mathcal{E}}
\DeclareMathOperator{\bN}{\mathcal{N}}

\DeclareMathOperator{\erfc}{erfc}

\DeclareMathOperator{\dom}{dom}
\DeclareMathOperator{\Law}{Law}

\DeclareMathOperator{\Lop}{L}

\DeclareMathOperator*{\sgn}{sgn}

\DeclareMathOperator*{\supp}{supp}
\DeclareMathOperator*{\Lip}{Lip}

\DeclareMathOperator*{\MC}{MC}
\DeclareMathOperator*{\acc}{acc}
\DeclareMathOperator*{\rej}{rej}

\newcommand{\rd}{\mathrm{d}}
\newcommand{\vd}{\,\mathrm{d}}
\renewcommand{\d}[1]{\ensuremath{\operatorname{d}\!{#1}}}

\newcommand{\process}[1]{(#1)_{t\ge 0}}

\newcommand{\indic}[1]{\mathbbm{1}_{#1}}
\newcommand{\bg}{{\bf \underline g}}

\newcommand{\loct}[3]{L^{#2}_{#3}(#1)} 
 
\newcommand{\occt}[3]{A^{#2}_{#3}(#1)} 

\newcommand{\hfprocesstrue}[2]{{#1}_{\frac{{#2}}{n}}}
\newcommand{\hfprocess}[3]{{#1}^{#2}_{\frac{{#3}}{n}}}

\newcommand{\sqn}[0]{ \sqrt{n}} 
\newcommand{\sqt}[0]{ \sqrt{t}} 
\newcommand{\un}[0]{ u_n}

\newcommand{\wtB}[0]{ \widetilde{B}} 
\newcommand{\wtX}[0]{ \widetilde{X}}

\newcommand{\braces}[1]{ ( #1 ) } 
\newcommand{\bigbraces}[1]{ \big( #1 \big) } 
\newcommand{\Bigbraces}[1]{ \Big( #1 \Big) } 
\newcommand{\biggbraces}[1]{ \bigg( #1 \bigg) } 

\newcommand{\sqbraces}[1]{ [1 ] } 
 
\newcommand{\Bigsqbraces}[1]{ \Big[ #1 \Big] } 
\newcommand{\biggsqbraces}[1]{ \bigg[ #1 \bigg] }

\newcommand{\bigcubraces}[1]{ \big\{ #1 \big\}} 
\newcommand{\Bigcubraces}[1]{ \Big\{ #1 \Big\}}

\newcommand{\absbraces}[1]{ | #1  |} 
\newcommand{\bigabsbraces}[1]{ \big| #1 \big|} 
\newcommand{\Bigabsbraces}[1]{ \Big| #1 \Big|}

\newcommand{\norm}[1]{ \| #1  \|}

\newcommand{\qvsingle}[1]{ \langle #1 \rangle }

\newcommand{\Bigqvsingle}[1]{\Big \langle #1 \Big \rangle }

\newcommand{\convergence}[1]{\xrightarrow[n\rightarrow \infty]{#1}}

\newcommand{\Top}[2]{ T_{#1}[#2] }

\newcommand{\Plimsup}[2]{ #1\!\text{-}\!\limsup_{#2} }
\newcommand{\Pliminf}[2]{ #1\!\text{-}\!\liminf_{#2} }

\newcommand{\na}[0]{\un }
\newcommand{\nta}[0]{u^{2}_{n} }

\usepackage{authblk}

\newcommand{\email}[1]{\href{mailto:{#1}}{{#1}}}

\title{Functional convergence to the local time\\ of a sticky diffusion}

\author[1]{Alexis Anagnostakis\thanks{\email{alexis.anagnostakis@univ-lorraine.fr}}}
\affil[1]{Université de Lorraine, CNRS, Inria, IECL, F-54000 Nancy, France}

\date{}

\begin{document}

	\maketitle
	
	\vspace*{-1cm}
	
	\begin{abstract}
		We establish the consistency of a local time approximation of a diffusion at a sticky threshold based on high-frequency observations. 
		First, we prove the result for sticky Brownian motion, and then extend it to Itô diffusions with a sticky point (SID). 
		For this, we derive the pathwise formulation of an SID along with respective versions of key stochastic calculus results (Itô formula, Girsanov theorem). Based on the local time approximation, we develop a consistent estimator for the stickiness parameter. We conclude with numerical experiments and assess  statistical properties of the stickiness estimator and the local time approximation.
	\end{abstract}
	
	\renewcommand{\thefootnote}{}
	
	\footnotetext{ 
		\textit{This version:} \date{\today} (post--publication) 
	}
	
	\renewcommand{\thefootnote}{\arabic{footnote}}

	\section{Introduction and main results}

The last decades we are witnessing the appearance of local time approximations and their applications in statistical estimation problems.
In \cite{Jac98}, it is proven that for an Itô diffusion $X$, $\ell \in \IR $, $g$ an integrable function and
$(\un)_n $ a sequence of $(0,\infty) $ such that
\begin{equation}\label{eq_un_condition}
	\begin{aligned}
		\lim_{n\ra \infty} \un/n &= 0,
		&		
		\lim_{n\ra \infty} \un &= \infty,
	\end{aligned}
\end{equation}
then the high-frequency statistics 
\begin{equation}\label{eq_functional_form}
	\begin{aligned}
	&\frac{\un}{n}\sum_{i=1}^{[nt]} g(\un \hfprocess{X}{}{i-1}  -\ell), 
	& n&\in \IN,
	\end{aligned}
\end{equation}
converges in probability to $\lambda(g)\loct{X}{\ell}{t} $, as $n\lra \infty $, where $\loct{X}{\ell}{}$ is the local time of $X$ at $\ell $ and $\lambda(g) = \int_{\IR}g(x)\vd x $.
Since then, similar results were proven in the case of the skew and the oscillating Brownian motions \cite{Lejay2014,Maz19}, stable processes \cite{Jeganathan2004} and the fractional Brownian motion \cite{Jaramillo2021,Podolskij2017}.
These gave rise to the usage of local time statistics \cite{Lejay2014,Lejay2019,Lejay2018}, \ie{}  statistics based on these approximations. 
In this paper we prove that under additional assumptions on $((\un)_n, g)$, the same result holds for Itô diffusions with a sticky point at $0$ at $\ell=0 $.
We then define a local time statistic that we prove to be a consistent estimator of the stickiness parameter.

Introduced by Feller in \cite{Fel52}, one-dimensional sticky diffusions are continuous processes that satisfy the strong Markov property and spend positive amount of time at some points of their state-space $\II $.
These points, called sticky points, can be located either in the interior or at an attainable  boundary of $\II $ (sticky reflection).
A diffusion has a sticky point at $0 $ iff its speed measure $m$ has an atom at $0$, \ie{} $m(\{0\})>0 $.
The mass of that atom $\rho = m(\{0\}) $ is called stickiness parameter, which dictates how much time the process spends at $0$ and there are no references known to us for its estimation.

Sticky processes have been recently used to model phenomena in finance, biology, quantum and classical mechanics.
In particular, they can be used to describe the behavior of interest rates near zero~\cite{KabKijRin,Nie2}, the behavior of molecules near a membrane~\cite{Gra95}, the concentration of pathogens in a healthy individual \cite{CalFar}, the dynamics of mesoscale particles upon contact in colloids \cite{KalHol,Stell91} and the motion of quantum particles when they reach a source of emission \cite{DavTru}. 
From a theoretical standpoint, they are used to create new types of probabilistic couplings \cite{EbeZim} and appear as the limit of storage processes \cite{HarLem}.
Many papers have appeared recently that address the numerical challenges of simulating sticky diffusions \cite{Ami,anagnostakis2023general,AnkKruUru3,BouRabee2020,MeoLiGon}.

The scope of this paper are sticky Itô diffusions or SID, \ie{} diffusions that have a homogeneous SDE dynamic away from a point, where they exhibit stickiness. 
For simplicity we suppose the point of stickiness to be located at $0$.
We will call $ (\rho,\mu,\sigma)$-SID the sticky Itô diffusion that solves the SDE of drift $\mu $ and volatility $\sigma $ and has stickiness $\rho>0 $ at $0$.

The most elementary sticky Itô diffusion is the $(\rho,0,1) $-SID,  called the sticky Brownian motion (see \cite{EngPes}). 
It is the process that has a Brownian dynamic away from $0$ and a sticky point at $0$ of stickiness $\rho $.

We suppose that $(\mu,\sigma)$ satisfy the following conditions.
The first guarantees the uniqueness in law, the second prevents any oscillating phenomena at $0$, the last  allows the usage of Girsanov theorem in Section \ref{ssec_proof_of_thm12a}.
\begin{condition}\label{cond_drift_volatility_functions}
	Let $B$ a standard Brownian motion defined on a probability space
	$(\Omega, \process{\bF_t},\Prob_x) $.
	We consider the following SDE:
	\begin{equation}\label{eq_classic_SDE_formulation}
		\vd X_t = \mu (X_t) \vd t + \sigma(X_t) \vd B_t,
	\end{equation}
	and $(\sigma,\mu) $ such that
	\begin{enumerate}
		\item uniqueness in law holds for \eqref{eq_classic_SDE_formulation}, \ie{} it  defines an Itô diffusion,
		
		\item $\sigma$ is positive and $C^{1}(\II) $,
		
		\item if $X$ solves \eqref{eq_classic_SDE_formulation}, $\Prob_x(X_0 = x)=1 $ and $\theta = \process{ \sigma'(X_t) - \frac{\mu(X_t)}{\sigma(X_t)}}  $, then for all $x\in \II $, 
		\begin{equation}
			\Esp_x \Bigsqbraces{\exp \Bigbraces{\int_{0}^{t} \theta_s\d B_s - \frac{1}{2} \int_{0}^{t} \theta^2_s\d s}} = 1. 
		\end{equation}
	\end{enumerate}
\end{condition}

A sticky Itô diffusion $X$ is a continuous semimartingale (see Corollary \ref{cor_SID_semimartingale}). 
As such, if defined on a probability space $ (\Omega,\process{\bF_t},\Prob_x)$, its  local time at $\ell$ can either be defined (see \cite[Section~VI.1]{RevYor}) 
\begin{enumerate}
	\item  as the continuous, strictly-increasing process $\loct{X}{\ell}{} $
	such that for all $t\ge 0 $,
	\begin{equation}\label{eq_local_time_first_definition}
		|X_t - \ell| - |X_0 - \ell| 
		= \int_{0}^{t} \sgn(X_s - \ell) dX_s + \loct{X}{\ell}{t},
	\end{equation}
	where $\sgn(x) = \indic{x \ge 0} -\indic{x < 0} $,
	\item if $\sigma $ is continuous at $\ell $, for every $t\ge 0 $, as the almost sure limit
	\begin{equation}\label{eq_local_time_alternative_definition}
		\begin{aligned}
		\Prob\text{-}\as, &
		& \loct{X}{\ell}{t} &= \lim_{\epsilon\ra 0} \frac{1}{\epsilon} \int_{0}^{t} \indic{0 \le X_s - \ell <\epsilon} \vd \qvsingle{X}_s
		= \lim_{\epsilon\ra 0} \frac{1}{2\epsilon} \int_{0}^{t} \indic{| X_s - \ell | <\epsilon} \vd \qvsingle{X}_s,
		\end{aligned}
	\end{equation}
	If $\sigma(\ell-)\not =  \sigma(\ell+)$, the limits in \eqref{eq_local_time_alternative_definition} are not equal.
	In this case, the notions of right-,left- and symmetric-local time appear.
	The fact that $\ell $ is not a skew, nor an oscillation threshold for $X$ prevents this from happening.
\end{enumerate}

If $X$ is a diffusion on natural scale, its local time can also be defined as follows (see \cite[p.174, p.183]{ItoMcKean96}).
The local time of $X$ 
is the random field $\{\loct{X}{y}{t};t\ge 0,y\in \II \} $ such that
\begin{equation}\label{eq_intro_diffusion_occtimes_formula}
	\begin{aligned}
		&\Prob\text{-}\as,
		&\int_{0}^{t}\indic{X_s \in A}\vd s &= \frac{1}{2}\int_{A}\loct{X}{y}{t} m(\rd y),
		& \forall A&\in \mathcal B( \II),
		& \forall &t\ge 0,
	\end{aligned}
\end{equation}
where $m$ is the speed measure of $X$.

The object of this paper is to prove necessary conditions for the statistic \eqref{eq_functional_form} of a sticky Itô diffusions $X$ to converge to its local time.

In particular, we prove the following results:

\begin{theorem}\label{thm_main_result_stqSDE_intro}
	Let $X^{\rho}$  be a $(\rho,\mu,\sigma)$-SID, with state space $\II $,  defined on a family of probability spaces $(\Omega, \process{\bF_t},\Prob_x)_{x\in \II}$ such that for all $x\in \II $,
	$\Prob_x (X_0 = x) = 1 $ and
	$(\mu,\sigma) $ satisfy Condition~\ref{cond_drift_volatility_functions}.
	Let $g:\IR \mapsto \IR$ be a bounded Lebesgue-integrable function which vanishes on an open interval around $0$ and $(\un^{2})_n $ a sequence that satisfies \eqref{eq_un_condition}. 
	Then, 
	\begin{equation}\label{eq_auxiliary_main_result_sde}
		\frac{\un}{n} \sum_{i=1}^{[nt]} g(\un\hfprocess{X}{\rho}{i-1} )
		\convergence{} 
		\frac{\lambda(g)}{\sigma(0)} \loct{X^{\rho}}{0}{t}.
	\end{equation}
	locally uniformly in time, in 
	$\Prob_x $-probability.
\end{theorem}

\begin{corollary}[of Theorem~\ref{thm_main_result_stqSDE_intro} and Lemma~\ref{prop_occupation_time_conv}]\label{cor_stickiness_estimator_consistency}
	In the setting of Theorem~\ref{thm_main_result_stqSDE_intro}, we consider the statistic
	\begin{equation}\label{eq_statistic_stickiness}
		\widehat{\rho}_{n}(X) := 2 \frac{\lambda(g)}{\sigma(0)} \frac{1}{\un} \frac{\sum_{i=1}^{[nt]} \indic{\hfprocess{X}{}{i-1}=0}}{
			\sum_{i=1}^{[nt]} g(\un \hfprocesstrue{X}{i-1}) }.
	\end{equation}
	If $g $ is a  bounded integrable function which vanishes on an open interval around $0$ and $(\un^{2})_n $ a sequence that satisfies \eqref{eq_un_condition}, then
	$\widehat{\rho}_{n}(X) $
	is a, conditionally on $\mathcal L := \{\tau_0<t\} $, consistent estimator of $\rho $ as $n \lra \infty $, \ie{} if 
	$P_{x}^{\mathcal L}(.)= P(.| X_{0}=x, \mathcal L)$, then 
	$\widehat{\rho}_{n}(X)
	\lra \rho $ in $ P_{x}^{\mathcal L}$-probability as 
	$n\lra \infty $. 
\end{corollary}

The first result is the local time approximation.
The second is the consistency of a stickiness parameter estimator built upon the local time approximation.
The latter corresponds to the estimating function (see \cite{JacSor}) defined at stage $n$, for all $\rho>0$ by
\begin{equation}
	f_n(X,\rho) := \frac{\rho}{2}
	\Bigbraces{ \frac{\sigma(0)}{\lambda(g)}\frac{\un}{n} \sum_{i=1}^{[nt]} g(\un \hfprocesstrue{X}{i-1})}
	-  \frac{1}{n} \sum_{i=1}^{[nt]} \indic{\hfprocesstrue{X}{i-1} = 0}.
\end{equation}

\begin{remark}
	Both Theorem~\ref{thm_main_result_stqSDE_intro} and Corollary \ref{prop_occupation_time_conv} hold for $\un = n^{\alpha} $ with $\alpha\in (0,1/2) $
	as well as for $\un = \log (n) $.
\end{remark}

We prove these results first, for the sticky Brownian motion and then by extension, to sticky Itô diffusions.
The extension is done with stochastic calculus arguments which require a path-wise formulation of these processes.
While this paper was under review it came to the author's knowledge that the path-wise formulation of sticky Itô diffusions was treated in \cite{Sal2017}. 
In Section \ref{sec_thm2}, we provide another proof  that uses the speed measure and scale function characterization instead of the infinitesimal generator.
Thus, we avoid direct usage of the Hille-Yosida theorem.

The path-wise formulation of sticky Itô diffusions generalizes results in \cite{EngPes} for the sticky Brownian motion and \cite{Nie2} for the Ornstein-Uhlenbeck process with sticky reflection. 
The proofs in \cite{EngPes,Nie2} are based on the existence of a strong solution of the non-delayed process ($Y$ in the proof of
Theorem~\ref{prop_stichy_sde_solution_diffusion}).
Our proof, as the one in \cite{Sal2017} does not require this assumption.

\paragraph{Outline.}
In Section~\ref{sec_preliminary_results}, we prove several properties of the sticky Brownian motion, its local time and its semi-group.
In Section~\ref{sec_thm1b}, we prove Theorem~\ref{thm_auxiliary_main_result_sbm} which is  Theorem~\ref{thm_main_result_stqSDE_intro} for the sticky Brownian motion.
In Section~\ref{sec_thm2}, we give an analytical definition of sticky Itô diffusions, prove their path-wise formulation and two path-wise results (explicit versions for SID of Itô's formula and Girsanov theorem).
In section \ref{sec_proofs}, we prove Theorem~\ref{thm_main_result_stqSDE_intro} and Corollary \ref{cor_stickiness_estimator_consistency}.
Section~\ref{sec_num_exp} is dedicated to numerical experiments.

\paragraph{Acknowledgments.} 
The author thanks Antoine Lejay and Denis Villemonais for their supervision and mentoring along with Sara Mazzonetto for insightful discussions and advice on the subject. 
The PhD thesis of A. Anagostakis is supported by a scholarship from the Grand-Est Region (France).

\paragraph{Post-publication Addendum.}

The singular part of the speed measure, as given in equation~\eqref{eq_dm}, has been corrected from the originally published version.
The author thanks David Criens and his student, Sebastian Hahn, for bringing this error to his attention.

In addition, we have now correctly presented the occupation times formula for a diffusion on natural scale, following~\cite[p.174, p.183]{ItoMcKean96}, in equation~\eqref{eq_intro_diffusion_occtimes_formula}.
In the originally published version, this formula was mistakenly applied to all regular diffusions, following~\cite[\S II.2.13]{BorSal}. This led to the error in the singular part of the speed measure.

\section{Some preliminary results}\label{sec_preliminary_results}
In this section we prove several properties of the sticky Brownian motion that will be used in the latter sections. 
The first property is the time-scaling of the sticky Brownian motion which is the analogue of~\cite[Proposition I.1.10(iii)]{RevYor} for the standard Brownian motion.
Then, we prove a characterization of the local time of the sticky Brownian motion. 
Finally, we prove that  the transition kernel of the sticky Brownian motion is bounded by the transition kernel of the  non-sticky one.
The transition kernel of a diffusion process $X$ with speed measure $m$ is the  positive measurable function $p: \IR_{+}\times \II^{2} \mapsto \IR_{+} $ (see  \cite[\S II.1.4]{BorSal} and \cite[p.149]{ItoMcKean96}) such that for every bounded measurable function $h $,
\begin{equation}\label{eq_ptk_definition}
	\Esp_x\bigbraces{h(X_t)} = \int_{\IR} h(y)p(t,x,y) m(\rd y),
\end{equation}
where $m$ is the speed measure of $X$ (see \eqref{eq_sticky_Brownian_motion_characterization}).
The left hand side of \eqref{eq_ptk_definition} defines a family of operators indexed by time $\process{P_t} $ defined for every $h\in L^{\infty}(\IR)$ by
\begin{equation}\label{eq_semi_group_definition}
	P_t h(x) = \Esp_x\bigbraces{h(X_t)}.
\end{equation}
It turns out that $\process{P_t} $ is a family of linear operators over $C_b(\II) $ that characterize the law of $X$ (see Chapter 3.1 of \cite{RogWilV2}) and satisfy the semi-group property with respect to $t$, \ie{} $P_{t+s} = P_t P_s  $.
We thus call $\process{P_t} $ the semi-group of the diffusion $X$.

In our proofs we will use results on semimartingales.
The first one is the characterization of diffusions on natural scale in terms of a time-changed Brownian motion.

\begin{theorem}[\cite{RogWilV2}, Theorem~V.47.1]\label{thm_diffusion_as_time_changed_brownian_motion}
	Let $X = \process{X_t} $ be a diffusion on natural scale with speed measure $m$ defined on a probability space $\bP_x = (\Omega, \bF, \Prob_x)$ such that $\Prob_x(X_0 = x) =1 $. 
	Then, there exists a Brownian motion $B = \process{B_t} $ defined on an extension of $\bP_x $ such that for every $t\ge 0 $,
	\begin{equation}\label{eq_martingale_diffusion_characterization}
		X_t = B_{\gamma(t)},
	\end{equation}
	where $\gamma $ is the right inverse of $A(t) := \frac{1}{2}\int_{\IR} \loct{B}{y}{t} m(\rd y) $ and for all $y\in \IR $, $\loct{B}{y}{} $ is the local time of $B$ at $y$.
\end{theorem}

\begin{corollary}[\cite{RogWilV2}, p.277, Remark~(ii)]\label{cor_time_changed_brownian_as_diffusion_process}
	Let $m_X $ be a locally finite strictly positive measure on $\IR $ and $B $ a standard Brownian motion defined on a probability space $\bP_x = (\Omega, \bF, \Prob_x) $ such that $\Prob_x(B_0 = x)= 1 $.
	For every $y\in \IR $, let $\loct{B}{y}{} $ the time-continuous version of the local time of $B$ at $y$, $A_X(t) = \frac{1}{2}\int_{\IR} \loct{B}{y}{t}m_X(\rd y) $  and $\gamma_X(t) $ the right-inverse of $A_X(t) $. 
	Then, the process $X$ such that $X_t = B_{\gamma_X(t)} $ for every $t\ge 0 $ is a diffusion process on natural scale with speed measure $m_X(\rd x) $ and such that $\Prob_x \bigbraces{X_0 = x}= 1 $.
\end{corollary}

The following results, given as exercises in \cite{RevYor}, allow us to characterize the local time of re-scaled and time-changed processes respectively.

\begin{lemma}[\cite{RevYor}, Exercise~VI.1.23]\label{lem_space_changed_local_time}
	Let $X$ be a continuous semi-martingale and $f$ a strictly increasing difference of two convex functions.
	If $f(X) = \process{f(X_t)} $, then for every $a \in \II $ and $t \ge 0 $,
	\begin{equation}
		\loct{f(X)}{f(a)}{t}	= f'(a)\loct{X}{a}{t}
	\end{equation}
	almost surely, where $f' $ is the right-derivative of $f$.
\end{lemma}

\begin{lemma}[\cite{RevYor}, Exercise~VI.1.27]\label{lem_time_changed_local_time}
	Let $X$ be a semi-martingale with state-space $\II $, $T $ a time change
	and $Y$ is the time changed process such that $Y_t = X_{T(t)} $ for every $t\ge 0 $. 
	Then for every $a\in \II $ and $t\ge 0 $, 
	\begin{equation}\label{eq_local_time_time_scaling_prop}
		\loct{Y}{a}{t} = \loct{X}{a}{T(t)}
	\end{equation}
	almost surely.
\end{lemma}

The last result allows us to extend convergences in probability of the marginals to locally uniform convergence in time, in probability.

\begin{lemma}[\cite{JacPro}, (2.2.16)]\label{lem_uniform_in_time_in_probability_conv}
	Let $\{X^{n}\}_{n\in \IN} $ be a sequence of increasing processes and $X$ an almost surely continuous  bounded process.
	If $\{X^{n}\}_{n\in \IN} $ and $X$ are defined on the probability space $(\Omega, \process{\bF},\Prob ) $ and
	\begin{equation}\label{eq_convergence_in_prob_lem_2_6}
		X_{t}^{n} \convergence{} X_t
	\end{equation}
	in probability, for every $t \in D $ with $D$ a dense subset of $\IR_+ $.
	Then, the convergence is locally uniform in time, in probability.
\end{lemma}

\subsection{Time-scaling for the sticky Brownian motion}

\begin{lemma}
	Let $\bigcubraces{(\Omega, \bF, \process{\bF_t}, \Prob^{\rho}_{x} );x\in \IR, \rho \ge 0}$ be a family of filtered probability spaces 
	and $X^{\rho} = \process{X^{\rho}_t}$ a process defined on $(\Omega, \bF, \process{\bF_t})$ such that under $\Prob^{\rho}_{x} $ it is the sticky Brownian motion of stickiness parameter $\rho $ and $\Prob^{\rho}_{x}(X_{0}^{\rho} = x)=1 $. 
	Then, 
	\begin{multline}\label{eq_rescaled_process}
		\Law_{\Prob^{\rho}_x}\bigbraces{X^{\rho}_{ct} , 	\loct{X^{\rho}}{0}{ct}, \occt{X^{\rho}}{+}{ct}; t\ge 0 }\\
		=\Law_{\Prob^{\rho/\sqrt{c}}_{x}}\bigbraces{ \sqrt{c} X^{\rho/ \sqrt{c}}_{t}  , 	\sqrt{c} \loct{X^{\rho / \sqrt{c}}}{0}{t}, c\occt{X^{\rho/\sqrt{c}}}{+}{t} ; t\ge 0 },
	\end{multline}
	where $\Prob^{\rho}_{x}(X^{\rho/\sqrt{c}}_{0} = \sqrt{c}x)=1 $ and $(\loct{X^{\rho}}{0}{},\occt{X^{\rho}}{+}{}) $ and $(\loct{X^{\rho/\sqrt{c}}}{0}{},\occt{X^{\rho/\sqrt{c}}}{+}{}) $ are the local times, occupation times\footnote{The occupation time by $X$ of a measurable set $B \in \mathbb{B}(\IR)$ is the process $\occt{X}{B}{}$ defined for every $t\ge 0 $ by 
		\[ \occt{X}{B}{t} = \int_{0}^{t} \indic{X_s \in B} \vd s. \]} of $\IR_{+} $ pairs of $X^{\rho} $ and
	$X^{\rho / \sqrt{c}} $ respectively.
\end{lemma}

\begin{proof}
	From \cite{Touhami2021}, the joint density of $(X^{\rho}_t,\loct{X^{\rho}}{0}{t},\occt{X^{\rho}}{+}{ct}) $ is defined for every $(t,x,y,l,o) \in \IR_{+}\times\IR^{2} \times [0,\frac{2}{\rho}] \times [0,t] $ by
	\begin{multline}
		\Prob_x(X^{\rho}_t \in \vd y,\loct{X^{\rho}}{0}{t} \in \vd \ell,\occt{X^{\rho}}{+}{ct} \in \vd o) \\
		= q_{\rho}(t,x,y,l,o)\vd y \vd l \vd o  	   
		= h \Bigbraces{o- \rho \ell, \frac{\ell}{2}+x}h\Bigbraces{t-o, \frac{\ell}{2}-y} \vd y \vd l \vd o,
	\end{multline}
	where $h(t,x) $ is the function defined for every $t \ge 0 $ and $x\in \IR $ by
	\begin{equation}
		h(t,x) = \frac{|x|}{\sqrt{2 \pi}t^{3/2}} e^{-x^2 / 2t}.
	\end{equation}
	We observe that
	\begin{equation}
		h(ct,x) = c^{-1}h(t,x/\sqrt{c}).
	\end{equation}
	Thus,
	\begin{equation}
		\begin{aligned}
		q_{\rho}(ct,x,y,l,o) 
		&= h \Bigbraces{o- \rho \ell, \frac{\ell}{2}+x}h\Bigbraces{ct-o, \frac{\ell}{2}-y}
		\\ &= h \Bigbraces{c \Bigbraces{ \frac{o}{c} - \frac{\rho \ell}{c}},  \frac{\ell}{2}+x} 
		h \Bigbraces{c \Bigbraces{t - \frac{o}{c} }, \frac{\ell}{2}-y } 
		\\ &= c^{-2} h \Bigbraces{ \frac{o}{c} - \frac{\rho \ell}{c},  \frac{\ell}{2\sqrt{c}}+\frac{x}{\sqrt{c}}} 
		h \Bigbraces{t - \frac{o}{c} , \frac{\ell}{2\sqrt{c}}-\frac{y}{\sqrt{c}} }
		\\ &= c^{-2} q_{\rho/\sqrt{c}}\Bigbraces{t, \frac{x}{\sqrt{c}},\frac{y}{\sqrt{c}},\frac{l}{\sqrt{c}},\frac{o}{c}}
		\end{aligned}
	\end{equation}
	and
	\begin{equation}
		q_{\rho}(ct,x,y,l,o) \vd y \vd \ell \vd o =
		q_{\rho/\sqrt{c}}\Bigbraces{t, \frac{x}{\sqrt{c}},\frac{y}{\sqrt{c}},\frac{l}{\sqrt{c}},\frac{o}{c}} \vd \Bigbraces{\frac{y}{\sqrt{c}}} \vd \Bigbraces{\frac{\ell}{\sqrt{c}}}\Bigbraces{ \vd \frac{o}{c}},
	\end{equation}
	which finishes the proof.
\end{proof}

\begin{corollary}\label{cor_semi_group_scaling}
	Let $ X^{\rho} = \process{X^{\rho}_t} $ be the sticky Brownian motion of stickiness parameter $\rho >0 $ and $  \process{P^{\rho}_{t}} $ be its semi-group. Then, for every measurable function $ h: \IR \mapsto \IR$,
	\begin{equation}\label{eq_semi_group_scaling}
		P_{t}^{\rho \sqn} h(x \sqn) = 
		\Esp_x \bigbraces{h(\sqn \hfprocess{X}{\rho}{t}) } .
	\end{equation}
\end{corollary}

\subsection{An estimate on the sticky semi-group}

\begin{lemma}\label{prop_semi_group_bounds}
	Let $\process{P^{\rho}_t} $ be the semi-group of the sticky Brownian motion of parameter $\rho >0 $. 
	There exists a constant $K>0 $ that does not depend on $ \rho$ such that for every real-valued  function $h(x)$ such that $h(0)=0 $,
	\begin{equation}\label{eq_lem_bound_semigroup_A}
		|P^{\rho}_t h(x)| \le K \frac{\lambda(|h|)}{\sqrt{t}},
	\end{equation}
	for every $t> 0 $, where $\lambda(g) = \int_{\IR} g(x)\vd x $.
\end{lemma}

\begin{proof}
	Let $p_{\rho}(t,x,y) $ be the probability transition kernel of the sticky Brownian motion of parameter $\rho>0 $ with respect to its speed measure $m(\d y) = 2 \d y + \frac{\rho}{2}\delta_{0}(\d y)$.
	From \cite[p.108]{BorSal},
	\begin{equation}\label{eq_sticky_kernel_decomposition_A}
		p_{\rho}(t,x,y) = u_1(t,x,y) - u_2(t,x,y) + v_{\rho}(t,x,y),
	\end{equation}
	for every $x,y\in \IR $ and $t>0 $, where $u_1,u_2,v_{\rho} $
	are defined for all $\rho,t>0 $, $x,y\in \IR $ by
	\[
	\left\{
	\begin{array}{ll}
		u_1(t,x,y) &=  \frac{1}{ \sqrt{2 \pi t}}  e^{-(x-y)^2/2t} ,\\
		u_2(t,x,y) &= \frac{1}{ \sqrt{ 2 \pi t}} e^{-(|x|+|y|)^2/2t},\\
		v_{\rho}(t,x,y) &= \frac{2}{\rho} e^{4(|x|+|y|)/\rho + 8t/\rho^2} \erfc \Bigbraces{\frac{|x|+|y|}{\sqrt{2t}} + \frac{2\sqrt{2t}}{\rho}}
	\end{array}
	\right.
	\]
	and 
	$\erfc:= [x \lra (2/\sqrt{\pi})\int_{x}^{\infty}\exp(-\zeta^2)\vd \zeta] $
	is the Gaussian complementary error function.

		\noindent
		We observe that for $x,y\in \IR $ and $t\ge 0 $:  $\exp(-(|x|+|y|)^2/2t) < \exp(-(x-y)^2/2t) $ and
		\begin{equation}\label{eq_u_bound}
			u_1(t,x,y) - u_2(t,x,y) \le  \frac{1}{2 \sqrt{2 \pi t}} e^{-(x-y)^2/2t}.
		\end{equation}
		The Mills ratio of a Gaussian random variable (see \cite[p.98]{GriSti}) yields  $\erfc(x) \sim e^{-x^2}/x $. Thus, there exists a constant $K_{Mills}>0 $ such that
		\begin{equation}\label{eq_v_bound}
			v_{\rho}(t,x,y) \le K_{Mills} \frac{2\sqrt{2t}}{\rho \braces{|x|+|y|} + 8t} e^{-(|x|+|y|)^2/2t} \le K_{Mills} \frac{1 }{2 \sqrt{2 t}}  e^{-(x-y)^2/2t}.
		\end{equation}
		From \eqref{eq_u_bound} and \eqref{eq_v_bound}, for $K = 1 + K_{Mills} \sqrt{\pi}/2$,
		\begin{equation}\label{eq_sticky_kernel_bound}
			p_{\rho}(t,x,y) \le K \frac{1}{\sqrt{2\pi t}} e^{-(x-y)^2/2t}
		\end{equation}
		and
		\begin{equation}
			|P_t h(x)| \le \int_{\IR} |h(y)|p_{\rho}(t,x,y) \vd y  
			\le K \int_{\IR} |h(y)|\frac{1}{\sqrt{2\pi t}} e^{-(x-y)^2/2t} \vd y.
		\end{equation}
		Observing $e^{-(x-y)^2/2t} \le 1  $ yields \eqref{eq_lem_bound_semigroup_A}.
	\end{proof}

	\section{The case of the sticky Brownian motion}\label{sec_thm1b}

	\begin{theorem}\label{thm_auxiliary_main_result_sbm}
		Let $X^{\rho}$  be the sticky Brownian motion of stickiness $\rho>0 $, defined on a
		family of probability spaces $(\Omega, \process{\bF_t},\Prob_x)_{x\in \IR}$ such that for all $x\in \IR $
		$\Prob_x (X_0 = x) = 1 $.
		Let be $g:\IR \mapsto \IR$ be a bounded Lebesgue-integrable function which vanishes on an open interval around $0$, $(\un)_n $ a sequence that satisfies \eqref{eq_un_condition} and $T: \IR \mapsto \IR$ a continuously differentiable function such that for an $\epsilon >0 $:
		\begin{equation}\label{eq_condition_T}
			T(0)=0, \qquad T'(0)=1, \qquad \epsilon \le T'(x) \le 1/\epsilon, \qquad \bigabsbraces{T''(x)} \le 1/\epsilon,
		\end{equation}
		for every $x\in \IR $.
		For every such function $T$, let $g_n[T] $ be the sequence of functions such that:
		\begin{equation}\label{eq_gn_T_definition_sde}
			g_n[T](x)= g\bigbraces{\un T(\un^{-1} x)},
		\end{equation} 	
		for every $x $ and $n$.
		Then, 
		\begin{equation}\label{eq_auxiliary_main_result_SBM}
			\frac{\un}{n} \sum_{i=1}^{[nt]} g_n[T](\un\hfprocess{X}{\rho}{i-1} )
			\convergence{} 
			\frac{\lambda(g)}{\sigma(0)} \loct{X^{\rho}}{0}{t}.
		\end{equation}
		locally uniformly in time, in 
		$\Prob_x $-probability.
	\end{theorem}
	
	As the proof of Theorem~\ref{thm_auxiliary_main_result_sbm} is rather tedious, we have isolated parts of it in Lemmas~\ref{lem_functional_sup_bound}, \ref{lem_local_time_conv} and \ref{lem_semigroup_step_conv}.
	
	\begin{lemma}\label{lem_functional_sup_bound}
		Let $X^{\rho} $ be the sticky Brownian motion of stickiness $\rho>0 $ and $g$ be an integrable function such that $g(0)=0 $. Then, there exists a constant $K>0 $ such that
		for every $x\in \IR $, $t>0 $ and $n\in \IN $:
		\begin{equation}\label{eq_gn_dilatation_main_result}
			\Esp_x \Bigbraces{\sup_{s\le t} \big|\frac{u_n}{n} \sum_{i=1}^{[ns]}g\bigbraces{u_n \hfprocess{X}{\rho}{i-1}}  \big|  }  \le K \Bigbraces{\frac{u_n}{n}|g(u_n x)| + \lambda(|g|) \sqrt{t}}.
		\end{equation}
		
	\end{lemma}
	
	\begin{proof}
		We observe that
		\begin{equation}\label{eq_sum_gn_bounds_A}
			\begin{aligned}
			\Esp_x \Bigbraces{\sup_{s\le t} \big|\frac{u_n}{n} \sum_{i=1}^{[ns]}g\bigbraces{u_n \hfprocess{X}{\rho}{i-1}}  \big|  } 
			&\le \Esp_x \Bigbraces{ \frac{u_n}{n} \sum_{i=1}^{[nt]}\bigabsbraces{g\bigbraces{u_n \hfprocess{X}{\rho}{i-1}}  }  } \\
			&= \frac{u_n}{n}|g(u_n x)| + \frac{u_n}{n} \sum_{i=2}^{[nt]} \Esp_x \bigbraces{\bigabsbraces{g(\un \hfprocess{X}{\rho}{i-1})}}.
			\end{aligned}
		\end{equation}
		Then, if $h_n(x) = g(u_n x/ \sqn) $,
		\begin{equation}\label{eq_gn_dilatations}
			\begin{aligned}
			\sum_{i=1}^{[nt]}h_n\bigbraces{\sqn \hfprocess{X}{\rho}{i-1}} &= 
			\sum_{i=1}^{[nt]}g\bigbraces{u_n \hfprocess{X}{\rho}{i-1}}, 
			& \lambda(|h_n|) &= \frac{\sqn}{u_n}\lambda(|g|), &  h_n(0)&= 0.
			\end{aligned}
		\end{equation}
		From \eqref{eq_sum_gn_bounds_A} and \eqref{eq_gn_dilatations},
		\begin{equation}\label{eq_sum_gn_bounds_B}
			\Esp_x \Bigbraces{\sup_{s\le t} \big|\frac{u_n}{n} \sum_{i=1}^{[ns]}g\bigbraces{u_n \hfprocess{X}{\rho}{i-1}}  \big|  } 
			\le \frac{u_n}{n}|g(u_n x)| + \frac{u_n}{n} \sum_{i=2}^{[nt]} \Esp_x \bigbraces{\bigabsbraces{h_n(\sqn \hfprocess{X}{\rho}{i-1})}}.
		\end{equation}
		As $h_n(0)=0 $, from \eqref{eq_semi_group_scaling} and \eqref{eq_lem_bound_semigroup_A},
		\begin{equation}\label{eq_Ex_widehatgn_semi_group_bound}
			\Esp_x \bigbraces{\bigabsbraces{h_n(\sqn \hfprocess{X}{\rho}{i-1})}} =  P_{i-1}^{\rho \sqn} | h_n(x \sqn) | \le K \frac{\lambda(|h_n|)}{\sqrt{i-1}}.
		\end{equation}
		From \eqref{eq_gn_dilatations}, \eqref{eq_Ex_widehatgn_semi_group_bound} and as $\sum_{i=1}^{[nt]} \frac{1}{\sqrt{i}} \le 2 \sqrt{nt}  $,
		\begin{equation}\label{eq_sum_gn_bounds_C}
			\sum_{i=2}^{[nt]} \Esp_x \bigbraces{\bigabsbraces{h_n(\sqn \hfprocess{X}{\rho}{i-1})}} \le
			2K  \lambda(|h_n|) \sqrt{nt} = 2K  \frac{n}{u_n} \lambda(|g|) \sqrt{t}. 
		\end{equation}
		From \eqref{eq_sum_gn_bounds_B} and \eqref{eq_sum_gn_bounds_C} we get \eqref{eq_gn_dilatation_main_result}.
	\end{proof}

	\begin{lemma}\label{lem_local_time_conv}
		Let $X $ be a sticky Brownian motion with local time $\loct{X}{a}{} $.
		Moreover, let $g$	and $T $ be two real-valued functions such that $g $ is bounded and integrable and $T $ satisfies \eqref{eq_condition_T}.
		Then, for any $t\ge 0 $,
		\begin{equation}\label{eq_loct_integral_convergence_at_zero}
			\int_{\IR} g_n (x) \loct{X}{x/\na}{[nt]/n}\d x \convergence{} \lambda(g) \loct{X}{0}{t},
		\end{equation}
		where $g_n $ is given by \eqref{eq_gn_T_definition_sde}.
	\end{lemma}

	\begin{proof}
		From Trotter's theorem \cite{Tro}, the local time of the standard Brownian motion $ \loct{B}{x}{}$ admits a version that is  $(t,x)$-jointly continuous.
		As the time-change $\gamma $ also admits a continuous version,
		from \eqref{eq_local_time_time_scaling_prop}, the local time $\loct{X}{x}{} $ of $X$
		admits a version that is  $(t,x)$-jointly continuous.
		Thus, 
		\begin{equation}\label{eq_local_time_pointwise_convergence}
			|\loct{X}{x/\na }{[nt]/n} -  \loct{X}{0}{t}| \convergence{} 0,
		\end{equation} 
		for every $t\ge 0$ and $x\in \IR $.
		Moreover, there exists a positive random variable $U$ such that
		\begin{equation}\label{eq_local_time_unif_boundedness}
			|\loct{X}{x/\na }{[nt]/n} -  \loct{X}{0}{t}| \le U,
		\end{equation}
		for every $x$ and $n$.
		Thus, as $g$ is bounded,
		\begin{equation}\label{eq_local_time_integral_decomposition}
			\Bigabsbraces{	\int_{\IR} g_n (x)\bigbraces{ \loct{X}{x/\na }{[nt]/n} - \loct{X}{0}{t} } \d x } 
			\le \norm{g}_{\infty} \int_{|x|\le q} \bigabsbraces{ \loct{X}{x/\na }{[nt]/n} - \loct{X}{0}{t} } \d x
			+ U \int_{|x|> q} \absbraces{g_n(x)} \d x.
		\end{equation}
		From \eqref{eq_local_time_pointwise_convergence}, \eqref{eq_local_time_unif_boundedness} and Lebesgue convergence theorem,
		\begin{equation}\label{eq_local_time_integral_decomposition_part1}
			\int_{|x|\le q} \bigabsbraces{ \loct{X}{x/\na }{[nt]/n} - \loct{X}{0}{t} } 
			\convergence{} 0.
		\end{equation}
		With a change of variables,
		\begin{equation}\label{eq_local_time_integral_decomposition_part2}
			\begin{aligned}
			\int_{|x|> q} \absbraces{g_n(x)} \d x &=   \int_{|x|> q} \absbraces{g(\na T(x/\na))} \d x  \\
			&= \int_{\un T(q/\un)}^{\infty} g(y) \frac{1}{T'(y/\un)} \vd y + \int_{-\infty}^{\un T(-q/\un)} g(y)  \frac{1}{T'(y/\un)} \vd y.
			\end{aligned}
		\end{equation}
		From \eqref{eq_condition_T} and \eqref{eq_local_time_integral_decomposition_part2},
		\begin{equation}
			\limsup_n \int_{|x|> q} \absbraces{g_n(x)} \d x \le 
			\frac{1}{\epsilon} \Bigbraces{ \int_{q}^{\infty} g(y)  \vd  y 
				+ \int_{-\infty}^{-q} g(y)   \vd y} 
		\end{equation}
		which since $g$ is integrable converges to $0$ as $q \rightarrow \infty $.
		From \eqref{eq_local_time_integral_decomposition}, \eqref{eq_local_time_integral_decomposition_part1} and \eqref{eq_local_time_integral_decomposition_part2},
		\begin{equation}\label{eq_gn_full_local_time_convergence}
			\Bigabsbraces{	\int_{\IR} g_n(x)\bigbraces{ \loct{X}{x/\na}{[nt]/n} - \loct{X}{0}{t} } \d x } \convergence{} 0.
		\end{equation}
		Using again the same change of variables as in \eqref{eq_local_time_integral_decomposition_part2},
		\begin{equation}\label{eq_gn_integral_bound_ito_g}
			\int_{\IR} g_n(x)\vd x = \int_{\IR} g\bigbraces{\na T(x/\na)} \vd x 
			= \int_{\IR} g(x) \frac{1}{T' \bigbraces{T^{-1}(x/\na)}} \vd x.
		\end{equation}
		Thus, as $g $ is integrable and $T'(x) \ge \epsilon $ for every $x \in \IR $, from Lebesgue convergence theorem,
		\begin{equation}\label{eq_gn_integral_convergence}
			\int_{\IR} g_n(x)\vd x \convergence{} 	\int_{\IR} g(x)\vd x.
		\end{equation}
		Equations \eqref{eq_gn_full_local_time_convergence} and \eqref{eq_gn_integral_convergence} yield
		\eqref{eq_loct_integral_convergence_at_zero}.
	\end{proof}

	\begin{lemma}\label{lem_semigroup_step_conv}
		Let $T_{n}$ be the functional defined for each real-valued function $h$ and $x\in \IR $ by
		\begin{equation}
			\Top{n}{h}(x) = \int_{0}^{1} (P_{ \nta s/n}^{\na \rho } h(x) - h(x) )   \d s,
		\end{equation}
		where $\process{P^{\na \rho}_t} $ is the semi-group of the sticky Brownian motion of stickiness parameter $\na \rho $. 
		Then, for every bounded integrable Lipschitz function $k$ such that $k$ vanishes on an open interval around $0$,
		\begin{equation}\label{eq_Tnk_total_convergence}
			\lambda(|\Top{n}{k_{n}}|) \convergence{} 0,
		\end{equation} 
		where $k_{n}(x) = k(\na T(x/\na)) $.
	\end{lemma}
	
	\begin{proof}
		From Jensen's inequality,
		\begin{equation}\label{eq_bound_integral_TnKnp}
			\lambda(|\Top{n}{k_{n}}|) 
			\le \int_{0}^{1} \int_{\IR}  \bigabsbraces{P_{ \nta s/n}^{\na \rho } k_{n}(x) - k_{n}(x) }   \d x \d s.
		\end{equation}
		From \eqref{eq_sticky_kernel_bound},
		\begin{equation}
			\begin{aligned}
			& \bigabsbraces{P_{ \nta s /n}^{\na \rho } k_{n}(x) - k_{n}(x) }   
			\\
			&\qquad \le \int_{\IR} |k_{n}(y) - k_{n}(x)| p_{\na \rho}( \nta s /n,x,y) m(\d y) \\
			&\qquad=  \int_{\IR} |k_{n}(y) - k_{n}(x)| p_{\na \rho}(  \nta s /n,x,y) \d y 
			+ \frac{\na \rho}{2} |k_{n}(x)| p_{\na \rho}( \nta s /n,x,0)\\
			&\qquad\le K \Bigsqbraces{ \int_{\IR} |k_{n}(y) - k_{n}(x)| \frac{1}{\na \sqrt{2\pi  s/ n}} e^{-(x-y)^2 n/2 s \nta} \d y + |k_{n}(x)| \frac{\sqn \rho}{2 \sqrt{2 \pi s}} e^{-x^2 n/ 2  s\nta }},
			\end{aligned}
		\end{equation}
		where $ \frac{1}{\na \sqrt{2\pi  s/ n}} e^{-(x-y)^2 n/2 s \nta} $ is the probability density function of a Gaussian $\bN\bigbraces{x, \nta s/n} $.
		Thus from positive Fubini,
		\begin{equation}\label{eq_integral_Rnknp_bound_decomp}
			\begin{aligned}
			\int_{\IR} \bigabsbraces{P_{ \nta s/ n}^{\na \rho } k_{n}(x) - k_{n}(x) } \d x   	\le & K \biggsqbraces{  \Esp \Bigsqbraces{ \int_{\IR} \bigabsbraces{ k_{n}(x +\na \sqrt{  s/ n} Z) - k_{n}(x) } \d x } \\
				&+ \int_{\IR} |k_{n}(x)| \frac{\sqn \rho}{2 \sqrt{2 \pi s}} e^{-x^2 n / 2  s \nta}  \d x  }
			\end{aligned}
		\end{equation}
		where $Z \sim \bN(0,1) $ under $\Prob_x $.
		For the first additive term of right-hand side of \eqref{eq_integral_Rnknp_bound_decomp},
		with the same argument as \eqref{eq_gn_integral_bound_ito_g},
		\begin{equation}
			\int_{\IR} k_n(x) \vd x \le \frac{1}{\epsilon} \int_{\IR} k(x) \vd x
		\end{equation}
		Thus, for every $\omega \in \Omega $,
		\begin{equation}\label{eq_Tnk_uniform_bound}
			\int_{\IR} \bigabsbraces{ k_{n}(x +\na \sqrt{  s/ n} Z) - k_{n}(x) } \d x \le
			\int_{\IR} \bigabsbraces{k_{n}(x +\na \sqrt{  s/ n} Z)} \vd x 
			+ \int_{\IR} \bigabsbraces{k_{n}(x)} \vd x \le \frac{2}{\epsilon} \lambda(|k|). 
		\end{equation}
		Moreover since $T \in C^1 $ and $k$ is Lipschitz, $\Prob_x $-almost surely,
		\begin{equation}\label{eq_Tnk_pointwise_convergence}
			\begin{aligned}
			\bigabsbraces{ k_{n}(x +\na \sqrt{  s/ n} Z) - k_{n}(x) } &= 
			\bigabsbraces{ k(\na T(x/\na + \sqrt{  s/ n} Z)) - k(\na T(x/\na)) } 
			\\ &\le
			\bigabsbraces{ k(\na T(x/\na) + \sqrt{  s/ n} Z\|T'\|_{\infty}) - k(\na T(x/\na)) } 
			\\ &\le
			|k|_{\Lip}   \sqrt{  s/ n} Z \|T'\|_{\infty},
			\end{aligned}
		\end{equation}
		which converges to $0$ as $n \rightarrow \infty $.
		From \eqref{eq_Tnk_uniform_bound}, \eqref{eq_Tnk_pointwise_convergence} and Lebesgue convergence theorem,
		\begin{equation}\label{eq_convergence_Tnk_expectancy_term}
			\Esp \Bigsqbraces{ \int_{\IR} \bigabsbraces{ k_{n}(x +\na \sqrt{  s/ n} Z) - k_{n}(x) } \d x } \convergence{} 0.
		\end{equation}
		For the second additive term of right-hand side of \eqref{eq_integral_Rnknp_bound_decomp}:
		let $\delta > 0 $ be a positive real number such that $k(x)=0 $ for every $x\ge 0 $ such that $x \notin (\delta, 1/\delta) $. 
		From \eqref{eq_condition_T}, $T$ is strictly increasing, thus,
		\begin{equation}
			\delta \le \un T(x/\un) \le 1/\delta,
		\end{equation}
		is equivalent to
		\begin{equation}
			\un	T^{-1}( 	\delta / \un ) \le x \le \un T^{-1}(1 /\un \delta).
		\end{equation}
		From \eqref{eq_condition_T},
		\begin{equation}
			\begin{split}
				\liminf_{n} \un	T^{-1}( 	\delta / \un ) &\ge \delta \epsilon,\\
				\limsup_{n} \un T^{-1}(1 /\un \delta) &\le 1/\delta \epsilon. \\
			\end{split}
		\end{equation}
		Thus, there exists $n_0 \in \IN $ such that for every $n\ge n_0 $, $\supp{k_n}\subset (\epsilon \delta /2, 2/\epsilon \delta) $.
		Thus, since $k$ is bounded,
		\begin{equation}\label{eq_bound_Tnh_singular_term}
			\begin{aligned}
			\int_{\IR} |k_{n}(x)| \frac{\sqn \rho}{2 \sqrt{2 \pi s}} e^{-x^2n / 2  s \nta}  \d x 
			&\le
			2 \norm{k}_{\infty} \int_{\epsilon \delta /2}^{2 / \epsilon \delta} \frac{\sqn \rho}{2 \sqrt{2 \pi s}} e^{-x^2n / 2  s \nta}  \d x
			\\ &\le \frac{\rho \norm{k}_{\infty} }{ \sqrt{2 \pi s}} \frac{2}{\epsilon \delta} \sqn e^{-(n/\nta) (\epsilon \delta)^2 / 8s },
			\end{aligned}
		\end{equation}
		which converges to $0$ as $n\rightarrow \infty $.
		From \eqref{eq_bound_integral_TnKnp}, \eqref{eq_integral_Rnknp_bound_decomp}, \eqref{eq_convergence_Tnk_expectancy_term} and \eqref{eq_bound_Tnh_singular_term}, the convergence \eqref{eq_Tnk_total_convergence} is proven.
	\end{proof}
	
	\begin{proof}[Proof (of Theorem~\ref{thm_auxiliary_main_result_sbm})]
		Let $X $ be the sticky Brownian motion of parameter $\rho > 0 $, $\process{P_{t}^{\rho}} $ its semi-group and $\loct{X}{x}{} $ its local time at $x$.
		From the occupation times formula and the characterization of $\qvsingle{X^{\rho}}_t $ in \cite{EngPes},
		\begin{equation*}
			\int_{0}^{t} f(X_s) \indic{X_s \ne 0} \vd s  = \int_{\IR} f(y) \loct{X}{y}{t}\vd y.
		\end{equation*}	
		Thus, if $f(0)=0 $,
		\begin{equation}\label{eq_occ_tim_formula_special}
			\int_{0}^{t} f(X_s) \vd s =  \int_{\IR} f(y) \loct{X}{y}{t}\vd y.
		\end{equation}
		By applying consecutive change of variables and from \eqref{eq_occ_tim_formula_special}, 
		\begin{equation}\label{eq_gn_occupation_time_equality}
			\begin{aligned}
			\int_{\IR} g_n(x) \loct{X}{ x/ \na}{[nt]/n}\d x 
			&= \na \int_{\IR} g_n( \na x) \loct{X}{x}{[nt]/n}\d x 
			\\ &= \na \int_{0}^{[nt]/n} g_n(\na X_s) \d s
			= \frac{\na}{n} \int_{0}^{[nt]} g_n(\na X_{s/n}) \d s.
			\end{aligned}
		\end{equation}
		Thus,
		\begin{equation}\label{eq_vn_decomposition}
			\begin{aligned}
				\frac{\na}{n} \sum_{i=1}^{[nt]} g_n (\na \hfprocesstrue{X}{i-1}) 
				=&\frac{\na}{n} \sum_{i=1}^{[nt]} g_n (\na \hfprocesstrue{X}{i-1})  - 
				\frac{\na}{n} \int_{0}^{[nt]} g_n (\na X_{s/n}) \d s 
				+ \int_{\IR} g_n (x) \loct{X}{x/\na}{[nt]/n}\d x \\
				=& \sum_{i=1}^{[nt]} \frac{\na}{n}    \int_{0}^{1} \bigbraces{  g_n (\na X_{\frac{i-1}{n} }) -   g_n (\na X_{\frac{i-1}{n} + \frac{s}{n} })   }  \d s \\
				& +  \frac{\na}{n} \int_{0}^{nt - [nt]} \bigbraces{ 
					g_n (\na X_{\frac{ [nt]}{n} })   - g_n (\na X_{\frac{ [nt]}{n} + \frac{s}{n} }) } \vd s\\
				& +  \int_{\IR} g_n (x) \loct{X}{x/\na}{[nt]/n}\d x.\\
			\end{aligned}
		\end{equation}
		\noindent
		For the second additive term at the right hand side of \eqref{eq_vn_decomposition},
		\begin{multline}\label{eq_remainder_term_convergence}
			\Bigabsbraces{ \frac{\na}{n} \int_{0}^{nt - [nt]} \bigbraces{ 
					g_n(\na X_{\frac{ [nt]}{n} })   - g_n(\na X_{\frac{ [nt]}{n} + \frac{s}{n} }) } \vd s } \\
			\le 
			\frac{\na}{n} \int_{0}^{nt - [nt]} \bigabsbraces{ 
				g_n(\na  X_{\frac{ [nt]}{n} })   - g_n(\na X_{\frac{ [nt]}{n} + \frac{s}{n} }) } \vd s 
			\le 2\norm{g}_{\infty} \frac{\na}{n},
		\end{multline}
		which converges to $0$ as $n \rightarrow \infty $.
		
		\noindent
		For the first additive term at the right hand side of \eqref{eq_vn_decomposition}, let
		\begin{equation}
			\begin{aligned}
			A^{n}_{t} &:= \sum_{i=1}^{[nt]} \frac{\na}{n}    \int_{0}^{1} \bigbraces{ g_n (\na X_{\frac{i-1}{n} + \frac{s}{n} }) -  g_n (\na X_{\frac{i-1}{n} })   }  \d s,
			\\
			B^{n}_{t} &:= \sum_{i=1}^{[nt]} \frac{\na}{n}    \int_{0}^{1} \Esp_x \Bigbraces{ g_n (\na X_{\frac{i-1}{n} + \frac{s}{n} }) -  g_n (\na X_{\frac{i-1}{n} }) \Big| \hfprocesstrue{\bF}{i-1}   }  \d s.
			\end{aligned}
		\end{equation}
		As the cross terms have expectancy $0$, from Minkowski's inequality,
		\begin{equation}\label{eq_convergence_A_B}
			\begin{aligned}
			\Esp_x \bigbraces{ |	A^{n}_{t}  - B^{n}_{t}  |^2 } &\le 2 \frac{\nta}{n^2} \sum_{i=1}^{[nt]} \Esp_x  \Bigbraces{  \int_{0}^{1} \bigbraces{ g_n (\na X_{\frac{i-1}{n} + \frac{s}{n} }) -  g_n (\na X_{\frac{i-1}{n} })   }  \d s }^2 
			\\ &\le 4 \norm{g}_{\infty} \frac{\nta}{n} \frac{[nt]}{n},
			\end{aligned}
		\end{equation}
		which converges to $0$ as $n \rightarrow \infty $. 
		Thus, $A^{n}_{t}  - B^{n}_{t}  $ converges to $0$ in $L^2(\Prob_x) $ and consequently in $L^1(\Prob_x) $. 
		As such, proving that $A^{n}_{t} \convergence{} 0$ in $L^1(\Prob_x) $ is equivalent to proving that  $B^{n}_{t} \convergence{} 0$ in $L^1(\Prob_x) $. 
		To prove the latter, we define for each real-valued function $h$ the functional,
		\begin{equation}
			\Top{n}{h}(x) = \int_{0}^{1} (P_{ \nta s/n}^{\na \rho } h(x) - h(x) )   \d s.
		\end{equation}
		From \eqref{eq_rescaled_process},
		\begin{equation}
			B^{n}_{t} = \frac{\na }{n} \sum_{i=1}^{[nt]} \Top{n}{g_n}(\na \hfprocesstrue{X}{i-1} ).
		\end{equation}
		From \eqref{eq_gn_dilatation_main_result}, there exists a constant $K'>0$ that does not depend on $n$ or $\rho $ such that
		\begin{equation}
			\Esp_x \bigbraces{|B^{n}_{t}| } \le K' \bigbraces{\frac{u_n}{n} |\Top{n}{g_n}(\un x)| + \sqrt{t} \lambda(|\Top{n}{g_n}|)  }.
		\end{equation}
		From \eqref{eq_gn_T_definition_sde} and \eqref{eq_semi_group_definition}, for every $t>0$ and $\rho>0 $, 
		\begin{equation}
			|P^{\rho}_{t}g_n(x)|
			\le \norm{g}_{\infty},
		\end{equation}
		By taking $K'' = K'(2\norm{g}_{\infty} \vee 1)  $,
		\begin{equation}\label{eq_Bn_total_bound}
			\Esp_x \bigbraces{|B^{n}_{t}| } \le K'' \bigbraces{\frac{\na}{n} + \sqrt{t} \lambda(|\Top{n}{g_n}|)  }.
		\end{equation}
		Thus, as $0<\alpha <1/2$,  it remains to prove that $\lambda(|\Top{n}{g_n}|) \rightarrow 0 $.
		For this we use a Lipschitz approximation of $g$.
		In particular, as $g$ is bounded and in $L^1(\d x) $, for each $p$ it is possible to find a Lipschitz function $k_p $ such that, $k_p(0)=0 $ and $\lambda(|g - k_p|)< 1/p $ (see the proof of \cite[Lemma~4.5]{Jac98}). 
		Let $k^{n}_{p}(x) = k_{p}(\na T(x/\na)) $, from \eqref{eq_condition_T},
		\begin{equation}\label{eq_difference_gn_knp_bound}
			\lambda(|g_n - k^{n}_p|)< 1/p\epsilon.
		\end{equation}
		Let $p_{\rho}(t,x,y) $ be the sticky Brownian motion transition kernel given in \eqref{eq_sticky_kernel_decomposition_A}.
		As $p_{\rho}(t,x,y)=p_{\rho}(t,y,x) $ for every $x,y \in \IR $,
		\begin{equation}\label{eq_difference_semigroup_gn_knp_bound}
			\begin{aligned}
			\lambda(|P^{\rho}_t g_n- P^{\rho}_t k^{n}_p|) &\le	\int_{\IR} \int_{\IR} \bigabsbraces{g_n(y) - k^{n}_p(y)}p_{\rho}(t,x,y)   \vd y   \vd x  
			\\ &=
			\int_{\IR} \bigabsbraces{g_n(y) - k^{n}_p(y)} \int_{\IR}  p_{\rho}(t,y,x)  \vd x  \vd y  
			\\ & \le \int_{\IR} \bigabsbraces{g_n(y) - k^{n}_p(y)} \vd y = \lambda(|g_n-k^{n}_p|).
			\end{aligned}
		\end{equation}
		From \eqref{eq_difference_gn_knp_bound} and \eqref{eq_difference_semigroup_gn_knp_bound},
		\begin{equation}\label{eq_Tng_bounded_by_Tnkp}
			\lambda(|\Top{n}{g_n}|) \le \frac{2}{p\epsilon} + 	\lambda(|\Top{n}{k^{n}_p}|).
		\end{equation}
		Thus, from \eqref{eq_Tnk_total_convergence},
		\begin{equation}\label{eq_Tng_convergence}
			\lambda(|\Top{n}{g_n}|) \convergence{} 0 .
		\end{equation}
		From \eqref{eq_convergence_A_B}, \eqref{eq_Bn_total_bound} and \eqref{eq_Tng_convergence},
		\begin{equation}\label{eq_An_convergence}
			A^{n}_{t} = \sum_{i=1}^{[nt]} \frac{\na}{n}    \int_{0}^{1} \bigbraces{ g_n(\na X_{\frac{i-1}{n} + \frac{s}{n} }) -  g_n(\na X_{\frac{i-1}{n} })   }  \d s \convergence{\Prob_x} 0.
		\end{equation}
		From \eqref{eq_loct_integral_convergence_at_zero}, \eqref{eq_vn_decomposition},  \eqref{eq_remainder_term_convergence} and \eqref{eq_An_convergence}, for every $t\ge 0 $,
		\begin{equation}
			\frac{\un}{n} \sum_{i=1}^{[nt]} g_n(\un \hfprocesstrue{X}{i-1}) \convergence{\Prob_x} \lambda(g) \loct{X}{0}{t}.
		\end{equation}
		If $g $ is a positive function the processes $\frac{\un}{n} \sum_{i=1}^{[nt]} g_n(\un \hfprocesstrue{X}{i-1}) $ are non-decreasing with $\Prob_x$-almost surely, a continuous limit. Thus, from Lemma~\ref{lem_uniform_in_time_in_probability_conv}, the convergence is locally uniform in time, in probability.
		For an arbitrary $g$ satisfying the conditions of Theorem~\ref{thm_auxiliary_main_result_sbm}, let $g = g^{+} - g^{-} $, where $g^{+}(x) = \max\{g(x),0\} $ and $g^{+}(x) = \max\{-g(x),0\} $.
		Then as $g^{+} $ and $g^{-} $ are both positive function and thus,
		\begin{equation}
			\frac{\un}{n} \sum_{i=1}^{[nt]} g^{+}_n(\un \hfprocesstrue{X}{i-1}) \convergence{} \lambda(g^{+}) \loct{X}{0}{t}, 
			\qquad \frac{\un}{n} \sum_{i=1}^{[nt]} g^{-}_n(\un \hfprocesstrue{X}{i-1}) \convergence{} \lambda(g^{-}) \loct{X}{0}{t},
		\end{equation}
		locally uniformly in time, in $\Prob_x$-probability.
		Using the triangle inequality for the absolute value and the $L^{\infty}(0,t)$-norm, the locally uniform convergence of \eqref{eq_auxiliary_main_result_SBM} in $\Prob_x $-probability is proven.
	\end{proof}

	\section{Sticky Itô diffusions}\label{sec_thm2}

	\subsection{Analytical characterization}\label{ssec_SID_def}
	
	In this section, we define sticky Itô diffusions via scale and speed.
	In general, one can describe the law of any regular one-dimensional diffusion process $X$ with state-space $ \II$ by a pair $(s,m)$ comprised of a continuous increasing function $s: \II \mapsto \IR $ and a strictly 
	positive\footnote{
		We say that a measure $m$ on $(\II,\mathcal{B}(\II))$ is strictly positive iff $m\bigbraces{(a,b)}>0 $ for any open interval $(a,b)\subset \II $.
	}, locally finite measure $m $ over $\II $ (see \cite[Section~VII.3]{RevYor}).
	The function $s $ is called scale function and expresses the propensity of the process to follow a certain direction (up or down).
	It can be defined as the function such that for each $a<b $ such that $a,b \in \II $,
	\begin{equation}\label{eq_def_scale_function}
		\Prob_x (\tau_b< \tau_a)= \frac{s(x)-s(a)}{s(b)-s(a)},
	\end{equation}
	where $\tau_{\zeta} = \inf\{s>0 : X_s = \zeta\}  $.
	The measure $m(\rd x) $ is called speed measure and expresses the speed at which the process moves and in particular the speed at which the process exits compact intervals.
	It can be defined as the unique positive locally finite measure $m $ such that for each $a<b $ with $a,b \in \II $,
	\begin{equation}\label{eq_def_speed_measure}
		\Esp_x(\tau_{ab})=\int_{(a,b)}\dfrac{\big((s(y\wedge x))-s(a)\big)\big(s(b)- (s(y\vee x))\big)}{s(b)-s(a)} m(\rd y).
	\end{equation}
	
	We consider the SDE 
	\begin{equation}
		\label{eq_SDE_nonsticky}
		\vd X_t = \mu(X_t) \vd t + \sigma(X_t) \vd B_t,
	\end{equation}
	where $ (\mu,\sigma)$ is a pair of functions that satisfy Condition \ref{cond_drift_volatility_functions} and $B$ is a standard Brownian motion.
	Then, the scale function and speed measure $(s,m) $ of the solution to~\eqref{eq_SDE_nonsticky} are 
	(see \cite[p.17]{BorSal})
	\begin{equation}\label{eq_dm_nosticky}
		\begin{aligned}
		s'(x) &=e^{-\int_{a}^{x}\frac{2\mu (u)}{\sigma^2 (u)}du},
		& m(\rd x) &= \frac{1}{s'(x)}\frac{2}{\sigma^2(x)}\vd x,
		& x&\in \IR,
		\end{aligned}
	\end{equation}
	with $s'$ being the right-derivative of $s$.
	
	The $(s,m)$-formulation is particularly adapted for studying sticky one-dimensional diffusions. 
	A point $\zeta$ of the state space is sticky if and only if $m(\{\zeta \})>0$.
	Moreover, since the $(s,m)$-formulation is local, one can characterize a sticky process in $(s,m)$-terms by adding an atom to the speed measure of the non-sticky version of this process.
	For example, the sticky Brownian motion of stickiness $\rho >0 $ is the diffusion process defined through scale and speed $(s,m) $, where
	\begin{equation}\label{eq_sticky_Brownian_motion_characterization}
		\begin{aligned}
			s(x)&= x,   &m(\d x) &= 2\d x + \rho \delta_0 (\d x),
			& x&\in \IR. 
		\end{aligned}
	\end{equation}
	We note that $s_{0}(x) = x $ and $m_{0}(\d x) = 2 \d x $ are the scale function and speed measure of the standard Brownian motion.
	Similarly, the $(\rho,\mu,\sigma)$-SID solution is the diffusion process 
	with $s$ and $m$ given by 
	\begin{equation}\label{eq_dm}
		\begin{aligned}
		s'(x) &= e^{-\int_{a}^{x}\frac{2\mu (u)}{\sigma^2 (u)}du},
		&m(\rd x) &= \frac{1}{s'(x)}\frac{2}{\sigma^2(x)}\vd x + \frac{\rho}{s'(0)} \delta_0(\d x),
		& x&\in \IR.
		\end{aligned}
	\end{equation}
	
For simplicity, we assume $s(0)=0 $, as scale functions are defined up to an affine transformation. Indeed, from~\eqref{eq_def_scale_function},~\eqref{eq_def_speed_measure}, we observe that if $(s,m)$ is the scale and speed of a diffusion $X$, then all pairs
\begin{equation}
	\begin{aligned}
		&(\alpha s+ \beta, \frac{1}{\alpha}m),
		& \alpha&>0
		& \beta&\in \IR
	\end{aligned}
\end{equation}
are also scale and speed of $X$.

	\subsection{Path-wise formulation}\label{ssec_pathwise_description}
	
	In this section we prove that
	\begin{enumerate}
		\item  the solution of \eqref{eq_sde_part}-\eqref{eq_loc_tim_part} is a $(\rho,\mu,\sigma)$-SID,
		\item  if $X$ is a $(\rho,\mu,\sigma)$-SID defined on a probability space $\bP $, then there exists a Brownian motion $B$ defined on an extension of $\bP $ such that $X$ solves \eqref{eq_sde_part}-\eqref{eq_loc_tim_part}.
	\end{enumerate}
	
	This generalizes results found in \cite{EngPes},\cite{Nie2} for the sticky Brownian motion and
	the Ornstein-Uhlenbeck with sticky reflection respectively.
	The main challenge with generalizing these results is that if $\sigma$ is non-constant, then
	the quadratic variation of a solution of \eqref{eq_sde_part}-\eqref{eq_loc_tim_part} is non-trivial. 
	This can be surmounted using the interchangeability between integration and time-changes (see \cite[Proposition~V.1.5]{RevYor}).
	This requires that the non-delayed SDE \eqref{eq_classic_SDE_formulation} has a strong solution.
	
	To avoid this hypothesis, we will use another approach which consists in identifying the law of the time-changed process.

	\begin{theorem}\label{prop_stichy_sde_solution_diffusion}
		We consider the following system
		\begin{align}
			\vd X_t	&= \mu(X_t) \indic{X_t \ne 0} \vd t + \sigma(X_t) \indic{X_t \ne 0} \vd B_t  , \label{eq_sde_part}\\
			\indic{X_t = 0}	\vd t &= \frac{\rho}{2}\vd \loct{X}{0}{t} , \label{eq_loc_tim_part}
		\end{align}
		where $B $ is a standard Brownian motion, $ \loct{X}{0}{}  $ is the local time process of $X$ at $0$ and $(\mu,\sigma) $ are a pair of real-valued functions over $\II $ that satisfy Condition \ref{cond_drift_volatility_functions}.
		If $(X,B)$ jointly solve \eqref{eq_sde_part}-\eqref{eq_loc_tim_part}, then $X$ is the $(\rho,\mu,\sigma)$-SID.
	\end{theorem}
	
	\begin{proof}
		Let $(X,B) $ be a solution of \eqref{eq_sde_part}-\eqref{eq_loc_tim_part} defined on
		the probability space $\bP_x = (\Omega, \process{\bF_t}, \Prob_x)$, where $Y_0 = x $, $\Prob_x$-almost surely.
		Also, let $ \gamma$ the time-transform defined for every $t\ge 0 $ by
		\begin{equation}\label{eq_gamma_wt_integral_definition}
			\gamma(t) = \int_{0}^{t} \indic{X_s \ne 0} \vd s	
		\end{equation}
		and $A  $ the right-inverse of $\gamma $.
		Also, let $Y $ be the process defined for every $t\ge 0 $ by
		\begin{equation}\label{eq_wtY_definition}
			Y_t = X_{A(t)}.
		\end{equation}
		From the definition of $Y $, it holds that $\Prob_x(Y_0=x)= 1 $.
		From \eqref{eq_loc_tim_part}, if $\loct{X}{0}{}$ and $\loct{Y}{0}{}$ are the local times at $0 $ of $X$ and $Y$ respectively,
		\begin{equation}\label{eq_gammat_right_inverse_in_terms_of_A}
			\begin{aligned}
			t = \gamma\bigbraces{A(t)} &= \int_{0}^{A(t)}\indic{X_s \ne 0} \vd s = A(t) - \int_{0}^{A(t)}  \indic{X_s = 0} \vd s = A(t) - \frac{\rho}{2} \loct{X}{0}{A(t)} \\
			&= A(t) - \frac{\rho}{2} \loct{Y}{0}{t}.
			\end{aligned}
		\end{equation}
		The time-transform $A$ is continuous and strictly increasing, thus $\gamma $ is its proper inverse.\\
		From \eqref{eq_sde_part},
		\begin{equation}
			\begin{aligned}
			Y_t 	= X_{A(t)} &=
			X_{0} + \int_{0}^{A(t)} \mu(X_s)  \indic{X_s \ne 0} \vd s 
			+ \int_{0}^{A(t)} \sigma(X_s) \indic{X_s \ne 0} \vd B_s 
			\\
			&= X_{0} + \int_{0}^{A(t)} \mu(X_s)   \vd \gamma(s ) + \int_{0}^{A(t)} \sigma(X_s) \indic{X_s \ne 0} \vd B_s 
			\\
			&= X_{0} + \int_{0}^{t} \mu(Y_s)   \vd s + \int_{0}^{A(t)} \sigma(X_s) \indic{X_s \ne 0} \vd B_s.
			\end{aligned}
		\end{equation}
		Let $B^{1} $ be the process defined for every $t\ge 0 $ by
		\begin{equation}\label{eq_underlying_brownian_motion_0}
			B^{1}_t = \int_{0}^{A(t)} \indic{X_s \ne 0} \vd B_s 
		\end{equation}
		where,
		\begin{equation}\label{eq_levy_char_sbm}
			\qvsingle{B^{1}}_{t} = \int_{0}^{A(t)} \indic{\wtX_s \ne 0} \d s = \gamma(A(t)) = t.
		\end{equation}
		From Levy's characterization $B^{1} $ is a standard Brownian motion and as $\vd B^{1}_{t} = \indic{Y_s \ne 0} \vd B_{A(t)} $, 
		\begin{equation}
			Y_t = X_{0} + \int_{0}^{t} \mu(Y_s)   \vd s + \int_{0}^{t} \sigma(Y_s) \indic{Y_s \ne 0} \vd B_{A(s)} \\
			= X_{0} + \int_{0}^{t} \mu(Y_s)   \vd s + \int_{0}^{t} \sigma(Y_s) \vd B^{1}_s.
		\end{equation}
		Thus, from \cite[Proposition~VII.2.6]{RevYor} and  \cite[p.17]{BorSal}, $Y$ is the diffusion process defined through $s$ and $m$ where
		\begin{equation}\label{eq_dm_standard_sde}
			\begin{aligned}
				s'(x)&=e^{-\int_{a}^{x}\frac{2\mu (u)}{\sigma^2 (u)}du},
				&
				m_{Y}(\rd x) &= \frac{1}{s'(x)}\frac{2}{\sigma^2(x)}\vd x.
			\end{aligned}
		\end{equation}
		From~\cite[Exercise~VII.3.8]{RevYor}, the process $Y' = \process{s(Y_t)}$ is a diffusion on natural scale with speed measure $m_{Y'}(\rd x) = m_{Y} (s(\rd x )) 
		= \frac{2}{\sigma^{2}(x)} \vd x $.
		From Theorem~\ref{thm_diffusion_as_time_changed_brownian_motion}, there exists a Brownian motion $W$ such that,
		\begin{equation}\label{eq_nonstickysde_on_natural_scale_as_time_changed_bm}
			\begin{aligned}
				s(Y_t)&=W_{\gamma_{Y'}(t)}, &\forall t &\ge 0,
			\end{aligned}
		\end{equation}
		where $\gamma_{Y'}(t)$ is the right-inverse of $A_{Y'}(t)= \int_{\IR} \loct{W}{y}{t}m_{Y'}(\rd y)$.
		From \eqref{eq_wtY_definition} and \eqref{eq_nonstickysde_on_natural_scale_as_time_changed_bm},
		\begin{equation}
			\begin{aligned}
				s(X_t) &= W_{\gamma_{Y'}(\gamma (t))}, & \forall t &\ge 0.
			\end{aligned}
		\end{equation}
		From \eqref{eq_gammat_right_inverse_in_terms_of_A} and Lemma~\ref{lem_space_changed_local_time}, the right-inverse of $\gamma_{Y'}(\gamma (t))$ is
		\begin{equation}
			A(A_{Y'}(t)) = A_{Y'}(t) + \frac{\rho}{2} \loct{Y}{0}{A_{Y'}(t)}
			= A_{Y'}(t) + \frac{\rho}{2 s'(0)} \loct{W}{0}{t} = \frac{1}{2} \int_{\IR} \loct{W}{y}{t} \nu(\rd y),
		\end{equation}
		where 
		\begin{equation}
			\nu(\rd x) 
			= \frac{2}{\sigma^{2}(x)} \vd x + \frac{\rho}{s'(0)} \delta_0(\rd x).
		\end{equation}
		Thus, from Corollary \ref{cor_time_changed_brownian_as_diffusion_process}, $ s(X) = \process{s(X_t)} $ is a diffusion process on natural scale of speed measure
		$\nu $.
		As $s$ is continuous and invertible, $X $ is the diffusion process of scale
		and speed $(s,m) $, where
		\begin{equation}\label{eq_Y_sf_sm_wrt_X}
			\begin{aligned}
				m(\rd x) &:= \frac{1}{s'(x)}\frac{2}{\sigma^2(x)}\vd x + \frac{\rho}{s'(0)} \delta_0 (\rd x),
				& x&\in \IR.  
			\end{aligned}
		\end{equation}
		Hence, $X$ is a $(\rho,\mu,\sigma)$-SID.
		This completes the proof. 
	\end{proof}
	
	\begin{proposition}\label{prop_ssde_pathwise_characterization}
		Let $X$ be a $(\rho,\mu,\sigma)$-SID defined on a probability space $\bP_x = (\Omega, \process{\bF_t}, \Prob_x) $.
		Then, there exists a Brownian motion $W$ defined on an extension of $\bP_x $ such that, under $\Prob_x $,
		\begin{align}
			X_t	&= x + \int_{0}^{t} \mu(X_s) \indic{X_s \ne 0} \vd s +
			\int_{0}^{t} \sigma(X_s) \indic{X_s \ne 0} \vd W_s  , \label{eq_sde_part_integral}\\
			\int_{0}^{t} \indic{X_s = 0}	\vd s &= \frac{\rho}{2} \loct{X}{0}{t}. \label{eq_loc_tim_part_integral}
		\end{align}
	\end{proposition}

	\begin{proof}
		Let $A $ be the right-inverse of $\gamma = [t \mapsto \int_{0}^{t} \indic{X_s \ne 0} \vd s] $
		and $Y $ the defined by
		\begin{equation}\label{eq_Y_expressed_as_X}
			\begin{aligned}
				Y_{t} &= X_{A(t)},
				& t&\ge 0.
			\end{aligned}
		\end{equation}
		Let $(s_X,m_X )$ and $(s_Y,m_Y )$ be respectively the scale and speed pairs of $X$ and $Y$.
		
		From \eqref{eq_dm_standard_sde} and \eqref{eq_Y_sf_sm_wrt_X}, we have 
		\begin{equation}\label{eq_speed_measure_X_Y_equality}
			\begin{aligned}
				s_Y &= s_X,& m_Y &= m_X - \frac{\rho}{s'(0)} \delta_0.
			\end{aligned}
		\end{equation}
		From \eqref{eq_dm} and \eqref{eq_speed_measure_X_Y_equality},
		\begin{equation}\label{eq_dm_sde}
			\begin{aligned}
			s_{Y}'(x)&=e^{-\int_{a}^{x}\frac{2\mu (u)}{\sigma^2 (u)}du},
			&
			m_Y(\rd x) &= \frac{1}{s'(x)}\frac{2}{\sigma^2(x)}\vd x.
			\end{aligned}
		\end{equation}
		Thus, from \cite[Theorem~VII.3.12]{RevYor}, the infinitesimal generator $\Lop_Y $ of $Y$ is
		\begin{equation}
			\Lop_Y f(x) = \mu(x) f'(x) +\frac{1}{2} \sigma^2(x) f''(x),
		\end{equation}
		for every $x\in \II $ and $f\in \dom(\Lop_Y) $ where $\dom(\Lop_Y) = C^2(\II)$.
		From \cite[Theorem~VII.2.7]{RevYor}, there exists a Brownian motion $W^{1}$ on an extension of $(\Omega, \process{\bF_t}, \Prob_x) $ such that $Y$ almost surely solves,
		\begin{equation}\label{eq_proof_SDE_nonsticky}
			\vd Y_t = \mu(Y_t) \vd t + \sigma (Y_t) \vd W^{1}_t,
		\end{equation}
		for every $t\ge 0 $.
		Let $W $ be the process such that
		\begin{equation}\label{eq_proof_link_WW0}
			W_t = W^{1}_{\gamma(t) } + \int_{0}^{t} \indic{X_s = 0} \vd W^{0}_{s},
		\end{equation}
		where $W^{0}$ is independent of $W^{1} $.
		We observe that 
		\begin{equation}
			\qvsingle{W}_t
			= \qvsingle{W^{1}_{\gamma(\cdot) }}_t
			+
			\Bigqvsingle{\int_{0}^{\cdot} \indic{X_s = 0} \vd W^{0}_{s}}_t
			= \gamma(t) + (t - \gamma(t))=t
		\end{equation}
		Thus, from Lévy's characterization, $W$ is a standard Brownian motion and
		\begin{equation}\label{eq_proof_link_WW1}
			W_t =  W^{1}_{\gamma(t) } + \int_{0}^{t}\indic{X_s = 0} \vd W_s.
		\end{equation}
		
		From \eqref{eq_Y_expressed_as_X},\eqref{eq_proof_SDE_nonsticky},\eqref{eq_proof_link_WW1},
		\begin{equation}\label{eq_proof_SDE2_SDE_part}
			\begin{aligned}
			X_t
			&= x + \int_{0}^{\gamma(t)} \mu(Y_s) \vd s
			+
			\int_{0}^{\gamma(t)} \sigma (Y_s) \vd W^{1}_s
			= x + \int_{0}^{t} \mu(X_s) \vd \gamma(s)
			+ \int_{0}^{t} \sigma (X_s) \vd W^{1}_{\gamma(s)}\\
			&=
			x + 
			\int_{0}^{t} \mu(X_s) \indic{X_s\not = 0} \vd s
			+ \int_{0}^{t} \sigma (X_s) 
			\indic{X_s \not = 0} \vd W_s.
			\end{aligned}
		\end{equation}
		This qualifies $X$ as a semimartingale.
		
		From \cite[Exercise~VII.3.8]{RevYor}, the process $s(X) $ is on natural scale, with speed measure
		\begin{equation}
			m_{s(X)}(\rd x)
			= \frac{2}{\sigma^{2}(x)} \vd x + \frac{\rho}{s'(0)} \delta_0(\rd x).
		\end{equation}
		From the occupation times formula for diffusions on natural scale \eqref{eq_intro_diffusion_occtimes_formula} and Lemma~\ref{lem_space_changed_local_time}, 
		\begin{equation}\label{eq_proof_SDE2_loct_part}
			\int_{0}^{t}\indic{X_s=0} \vd s = \int_{0}^{t}\indic{s(X_s)=0} \vd s = \frac{\rho}{2s'(0)} \loct{s(X)}{0}{t} = \frac{\rho}{2} \loct{X}{0}{t}.
		\end{equation}
		Thus, $(X,W)$ jointly solve \eqref{eq_sde_part_integral}-\eqref{eq_loc_tim_part_integral}.
	\end{proof}
	
	\begin{corollary}\label{cor_SID_semimartingale}
		A sticky Itô diffusion is a semimartingale.
	\end{corollary}
	
	\begin{proof}
		Let $X$ be the $(\rho,\mu,\sigma) $-SID.
		From Proposition \ref{prop_ssde_pathwise_characterization}, there exists a Brownian motion $W$ such that $X$ solves \eqref{eq_sde_part_integral}-\eqref{eq_loc_tim_part_integral}.
		We observe that \eqref{eq_sde_part_integral} is the Doob-Meyer decomposition of $X$.
	\end{proof}
	
	\subsection{Additional results }\label{ssec_proof_of_thm12a}
	
	\begin{lemma}[sticky Girsanov]\label{thm_sticky_girsanov}
		Let $(\Omega, \process{\bF}, \Prob)$ be a probability space and $X $ the process that solves
		\begin{align}
			\vd X_t	&= \mu(X_t) \indic{X_t \ne 0} \vd t + \sigma(X_t) \indic{X_t \ne 0} \vd B_t  , \label{eq_sde_P}\\
			\indic{X_t = 0}	\vd t &= \frac{\rho}{2}\vd \loct{X}{0}{t}, \label{eq_loct_P}
		\end{align}
		where $B$ is a $\Prob $-Brownian motion.
		Let $\theta $ be a processes such that $\Prob(\int_{0}^{T} \theta_s \d s < \infty)= 1 $,  
		$\mathcal{E}(\theta) $ the process such that
		\begin{equation}
			\mathcal{E}_t(\theta) = 	\exp \Bigbraces{\int_{0}^{t} \theta_s\d B_s - \frac{1}{2} \int_{0}^{t} \theta^2_s\d s},
		\end{equation}
		for every $t\ge 0 $
		and $\Qrob $  the probability measure such that $\d \Qrob = \mathcal{E}_t(\theta) \d \Prob $.
		Then, if $\Esp_{\Prob} $ is the expectancy under $\Prob $ and $\Esp_{\Prob}\bigbraces{\bE_t(\theta)} = 1 $, the process $X $ solves
		\begin{align}
			\vd \wtX_t	&= \bigbraces{\mu(\wtX_t) + \theta_t \sigma(\wtX_t)} \indic{\wtX_t \ne 0} \vd t + \sigma(\wtX_t) \indic{\wtX_t \ne 0} \vd \widetilde{B}_t  , \label{eq_sde_Q}\\
			\indic{\wtX_t = 0}	\vd t &= \frac{\rho}{2}\vd \loct{\wtX}{0}{t}, \label{eq_loct_Q}
		\end{align}	
		where $\widetilde{B}_t = B_t - \int_{0}^{t} \theta_s \d s  $ is a standard Brownian motion under $\Qrob $.
	\end{lemma}
	
	\begin{proof}
		Let $X$ be the the solution of \eqref{eq_sde_P}-\eqref{eq_loct_P}, $\gamma $ the time-change  $\gamma(t) = \int_{0}^{t} \indic{X_s \ne 0} \vd s $ for every $t\ge 0 $, $A$ its right-inverse and $Y = \process{X_{A(t)}}$. 
		Let $\widetilde{B} $ be the process defined by $\widetilde{B}_t = B_t - \int_{0}^{t}  \theta_s \d s $ for every $t\ge 0 $.
		Then, from \cite[Theorem~6.3]{LipShi1}, $\widetilde{B} $ is a standard Brownian motion under
		$\Qrob  $ and the probability measures $\Prob $ and $\Qrob $ are equivalent.
		By substitution,
		\begin{equation}\label{eq_proof_girs_sde_part}
			\vd X_t	= \bigbraces{\mu(X_t) + \theta_t \sigma(X_t)} \indic{X_t \ne 0} \vd t + \sigma(X_t) \indic{X_t \ne 0} \vd \widetilde{B}_t.
		\end{equation}
		Moreover, since: the local time $\loct{X}{0}{t} $ and the quadratic variation $\qvsingle{X}_t $  are defined as limits in probability, $\Prob \sim \Qrob $ and
		\begin{equation}	
			\qvsingle{X}_t = \int_{0}^{t} \indic{X_t \ne 0}	\vd t = t - \int_{0}^{t} \indic{X_t = 0}	\vd t
		\end{equation}
		Thus, \eqref{eq_loct_P} holds also under $\Qrob $ and $X$ solves \eqref{eq_sde_Q}-\eqref{eq_loct_Q}.
	\end{proof}

	\begin{lemma}[sticky Itô formula]\label{lem_sticky_Ito}
		Let $X $ be a process defined on a probability space $(\Omega, \process{\bF_t},\Prob_x) $ such that $\Prob_x(X_0 = x)=1 $ and
		\begin{align}
			\vd X_t	&= \mu(X_t) \indic{X_t \ne 0} \vd t + \sigma(X_t) \indic{X_t \ne 0} \vd B_t  , \label{eq_sde_part_A}\\
			\indic{X_t = 0}	\vd t &= \frac{\rho}{2}\vd \loct{X}{0}{t}, \label{eq_loc_tim_part_A}
		\end{align}
		where $B$ is a $(\Omega, \process{\bF_t},\Prob_x) $-standard Brownian motion.
		Then, for every real valued $C^2 $ function $f $ such that  $f(0)=0 $ and $f'(0)\ne 0 $, the process $f(X) = \process{f(X_t)} $ is solution of
		\begin{align}
			\vd f(X_t)	&= \Bigbraces{f'(X_t) \mu(X_t) \frac{1}{2}f''(X_t) \sigma^2(X_t)}  \indic{X_t \ne 0} \vd t +  f'(X_t) \sigma(X_t)  \indic{X_t \ne 0} \vd B_t , \label{eq_sde_part_B}\\
			\indic{f(X_t) = 0}	\vd t &= f'(0) \frac{\rho}{2} \vd \loct{f(X)}{0}{t} \label{eq_loc_tim_part_B}
		\end{align}
		and $\Prob_x\bigbraces{f(X_0) = f(x)}=1 $.
	\end{lemma}
	
	\begin{proof}
		The process $X $ is a semi-martingale as,
		\begin{equation}
			X_t = X_0 + \int_{0}^{t} \mu(X_s) \indic{X_s \ne 0} \d s + \int_{0}^{t} \sigma(X_s) \indic{X_s \ne 0} \d B_s,
		\end{equation}
		where $\int_{0}^{t} \mu(X_s) \indic{X_s \ne 0} \d s  $ is a process of bounded variation and  $\int_{0}^{t} \sigma(X_s) \indic{X_s \ne 0} \d B_s $ is a local martingale.
		Thus, we may apply the standard Itô formula for $f\in C^2(\IR) $,
		\begin{equation}\label{eq_proof_applied_standard_ito_formula}
			\begin{aligned}
			\vd f(X_t)	&= f'(X_t)  \vd X_t + \frac{1}{2}f''(X_t) \vd \qvsingle{X}_t =  \\
			&= \Bigbraces{ f'(X_t) \mu(X_t)  + \frac{1}{2}f''(X_t) \sigma^2(X_t)} \indic{X_t \ne 0}  \d t + f'(X_t)\sigma(X_t)  \indic{X_t \ne 0} \d B_t,
			\end{aligned}
		\end{equation}
		thus proving \eqref{eq_sde_part_B}. 
		Lemma~\ref{lem_space_changed_local_time} and
		\eqref{eq_loc_tim_part_A} yield \eqref{eq_loc_tim_part_B}, thus the result is proven.
	\end{proof}

	\section{Proofs of the main results}\label{sec_proofs}

	\subsection{Proof of Theorem~\ref{thm_main_result_stqSDE_intro}}\label{ssec_proof_mainresult}

	\begin{proof}[Proof (of Theorem~\ref{thm_main_result_stqSDE_intro})]
		We suppose there exists a $\delta >0 $ such that
		\begin{equation}\label{eq_temporary_condition_sigma}
			\delta \le \sigma(x)\le 1/\delta,\qquad \bigabsbraces{\sigma'(x)} \le 1/\delta ,
		\end{equation}
		for every $x\in \II $.
		From Proposition \ref{prop_ssde_pathwise_characterization}, there exists a Brownian motion $B$ such that $(X,B)$ jointly solves
		\begin{align}
			\vd X_t	&= \mu(X_t) \indic{X_t \ne 0} \vd t + \sigma(X_t) \indic{X_t \ne 0} \vd B_t  , \label{eq_sde_part_D}\\
			\indic{X_t = 0}	\vd t &= \frac{\rho}{2}\vd \loct{X}{0}{t} . \label{eq_loc_tim_part_D}
		\end{align}
		Let $\Qrob_x  $ be the probability measure such that $\d \Qrob_x = \mathcal{E}_t(\theta) \d \Prob_x $ where
		\begin{equation}
			\begin{aligned}
			\mathcal{E}_t(\theta)  &= 	\exp \Bigbraces{\int_{0}^{t} \theta_s \d B_s - \frac{1}{2} \int_{0}^{t} \theta^2_s \d s},&
			\theta_t &= \sigma'(X_t) - \frac{\mu(X_t)}{\sigma(X_t)},
			& t&\ge 0.
			\end{aligned}
		\end{equation}
		From Lemma~\ref{thm_sticky_girsanov}, $(X, \wtB) $ jointly solve 
		\begin{align}
			\vd X_t	&= \frac{1}{2}\sigma(X_t)\sigma'(X_t) \indic{X_t \ne 0} \vd t + \sigma(X_t) \indic{X_t \ne 0} \vd \wtB_t  , \\
			\indic{X_t = 0}	\vd t &= \frac{\rho}{2}\vd \loct{X}{0}{t},
		\end{align}		
		where $\wtB = B_t - \int_{0}^{t} \theta_s \d s $ is a standard Brownian motion under $\Qrob_x$.
		Let $S$ be the function defined for every $ x\in \IR$ by
		\begin{equation}\label{eq_S_sigma_def}
			S(x) =\int_{0}^{x} \frac{1}{\sigma(y)} \d y.
		\end{equation}
		We observe that $S $ is strictly increasing and $S(0)=0 $.
		Thus, from \eqref{eq_sde_part_B}-\eqref{eq_loc_tim_part_B}, the process $X'= \process{S(X_t)}$ solves
		\begin{align}
			\d X'_t &= \frac{1}{\sigma(X_t)} \d X_t - \frac{1}{2} \frac{\sigma'(X_t)}{\sigma^2(X_t)} \vd \qvsingle{X}_t
			=  \indic{X'_t \ne 0}  \d \wtB_t. \label{eq_ito_applied_sde} \\
			\indic{X'_t = 0}	\vd t &= 
			\frac{\rho}{2\sigma(0)} \vd \loct{X'}{0}{t}.\label{eq_ito_applied_loctime}
		\end{align}
		From Theorem~\ref{prop_stichy_sde_solution_diffusion}, $X' $ is a sticky Brownian motion of parameter $\rho>0 $.
		Let $T$ be the function defined for all $x\in \IR$ by
		\begin{equation}\label{eq_space_transforms}
			T(x) = S^{-1}(x)/\sigma(0).
		\end{equation}
		We observe that
		\begin{equation}
			T'(x) = \frac{\sigma(x)}{\sigma(0)}.
		\end{equation}
		From \eqref{eq_temporary_condition_sigma},
		for every $x\in \IR $,
		\begin{equation}\label{eq_condition_T0_proof_0}
			\begin{split}
				&T(0)=0,\\ 
				&T'(0)=1,\\ 
				&\delta  /\sigma(0) \le T'(x)  \le 1/ \delta  \sigma(0), \\
				&\bigabsbraces{T''_0(x)} \le \norm{T''}_{\infty} \norm{\sigma}^{2}_{\infty} / \sigma^2(0) + \norm{T'}_{\infty} \norm{\sigma'}_{\infty}/ \sigma(0) \le \frac{1} { \delta \sigma(0)} \Bigbraces{\frac{1}{\delta \sigma(0)}+1 }.
			\end{split} 
		\end{equation}
		We observe that $T $
		satisfies \eqref{eq_condition_T} for $\epsilon =   \delta \bigbraces{\sigma(0)\wedge 1/\sigma(0)}$.
		Thus, from Theorem~\ref{thm_auxiliary_main_result_sbm} and \eqref{eq_ito_applied_sde}-\eqref{eq_ito_applied_loctime},
		\begin{equation}
			\frac{\un}{n} \sum_{i=1}^{[nt]} g_n[T]\bigbraces{\un \hfprocesstrue{X'}{i-1}} \convergence{} \lambda(g) \loct{X'}{0}{t},
		\end{equation}
		locally, uniformly in time, in $\Qrob_x $-probability, where $g_n[T](x) = g(\na T(x/\na)) $.
		From \eqref{eq_space_transforms} and the definition of $X' $,
		\begin{equation}\label{eq_functional_convergence_proof_v09}
			\frac{\na}{n} \sum_{i=1}^{[nt]} g(\na \hfprocesstrue{X}{i-1}) \convergence{} \lambda(g) \loct{X'}{0}{t},
		\end{equation}
		locally, uniformly in time, in $\Qrob_x $-probability.
		From \eqref{eq_local_time_alternative_definition}, the local time is defined as a limit in probability.
		Thus, from \eqref{eq_functional_convergence_proof_v09} and as $\Prob_x \sim \Qrob_x$,
		\begin{equation}\label{eq_result_for_bounded_sigma}
			\frac{\un}{n} \sum_{i=1}^{[nt]} g(\un \hfprocesstrue{X}{i-1}) \convergence{} \lambda(g) \loct{X'}{0}{t},
		\end{equation}
		locally, uniformly in time, in $\Prob_x $-probability.
		From Lemma~\ref{lem_space_changed_local_time}, $\loct{X}{0}{t} /  \sigma(0)  $ is a version of $\loct{X'}{0}{t}$.
		Thus, since we supposed \eqref{eq_temporary_condition_sigma}, the convergence  \eqref{eq_auxiliary_main_result_sde} is proven in the case of bounded $\sigma $, $1/ \sigma $, $\sigma' $. From \cite[Section~2.5]{Jac98}, the proof is extended to any $\sigma \in C^1 $.
	\end{proof}

	\subsection{Proof of Corollary \ref{cor_stickiness_estimator_consistency}}\label{ssec_proof_of_cor13}
	
	From \eqref{eq_loc_tim_part}, the occupation/local time ratio is the stickiness parameter.
	Thus having estimations of the two aforementioned quantities
	results in a consistent estimator of the stickiness parameter.
	We first show that the occupation time can be consistently approximated by a Riemann sum.
	Then, we use it along with Theorem~\ref{thm_main_result_stqSDE_intro} to prove Corollary~\ref{cor_stickiness_estimator_consistency}.
	
	\begin{lemma}\label{prop_occupation_time_conv}
		Let $X $ be a semi-martingale and $\occt{X}{0}{} $ be its occupation time of $0$ defined for every $t\ge 0 $ by
		\begin{equation}
			\occt{X}{0}{t} = \int_{0}^{t} \indic{X_s = 0} \vd s.
		\end{equation}
		Then,
		\begin{equation}\label{eq_prop_occupation_time_conv_result}
			\frac{1}{n}\sum_{i=1}^{[nt]}\indic{\hfprocess{X}{}{i-1}=0} \convergence{}  	\occt{X}{0}{t},
		\end{equation}
		locally uniformly in time, in probability.
	\end{lemma}
	
	\begin{proof}
		As both $t \mapsto \frac{1}{n}\sum_{i=1}^{[nt]}\indic{\hfprocess{X}{}{i-1}=0} $ and $	\occt{X}{0}{} $ are increasing processes, with the same argument as in the proof of Theorem~\ref{thm_auxiliary_main_result_sbm}, it suffices to prove the convergence in probability for each $t>0 $.
		If $\delta >0 $ and $\varepsilon >0 $ are two positive numbers,
		\begin{equation}\label{eq_excess_probability_decomposition}
			\begin{aligned}
			&\Prob_x \Bigbraces{ \bigabsbraces{\frac{1}{n} \sum_{i=1}^{[nt]}  \indic{\hfprocess{X}{}{i-1}=0} - 	\occt{X}{0}{t} } > \delta} \\
			&\qquad= \Prob_x \Bigbraces{ \bigabsbraces{\frac{1}{n} \sum_{i=1}^{[nt]}  \indic{|\hfprocess{X}{}{i-1}|<\varepsilon} -  \frac{1}{n} \sum_{i=1}^{[nt]}  \indic{0<|\hfprocess{X}{}{i-1}|<\varepsilon}  - 	\occt{X}{0}{t} } > \delta} \\
			&\qquad\le \Prob_x \Bigbraces{ \bigabsbraces{\frac{1}{n} \sum_{i=1}^{[nt]}  \indic{|\hfprocess{X}{}{i-1}|<\varepsilon}   - 	\occt{X}{0}{t} } + \bigabsbraces{\frac{1}{n} \sum_{i=1}^{[nt]}  \indic{0<|\hfprocess{X}{}{i-1}|<\varepsilon}} > \delta} \\
			&\qquad\le \Prob_x \Bigbraces{ \bigabsbraces{\frac{1}{n} \sum_{i=1}^{[nt]}  \indic{0<|\hfprocess{X}{}{i-1}|<\varepsilon}} > \frac{\delta}{2}} 
			+ \Prob_x \Bigbraces{ \bigabsbraces{\frac{1}{n} \sum_{i=1}^{[nt]}  \indic{|\hfprocess{X}{}{i-1}|<\varepsilon}   - 	\occt{X}{0}{t} }  > \frac{\delta}{2}}.
		\end{aligned}
		\end{equation}
		\noindent
		From \eqref{eq_lem_bound_semigroup_A} for $h(x)= \indic{0<|x|<\varepsilon} $ and as
		$\sum_{i=1}^{[nt]} \frac{1}{\sqrt{i}} \le 2 \sqrt{nt} $,
		\begin{equation}
			\Esp_x \Bigbraces{ \frac{1}{n}  \sum_{i=1}^{[nt]}  \indic{0<|\hfprocess{X}{}{i-1}|<\varepsilon} } 
			= \frac{1}{n} \sum_{i=1}^{[nt]} \Esp_x \bigbraces{ h(\hfprocess{X}{}{i-1})} \le \frac{1}{n} 2 \varepsilon K \sqn \sum_{i=1}^{[nt]} \frac{1}{\sqrt{i}} \le  4 \varepsilon K \sqt.  
		\end{equation}
		Thus, from Markov's inequality,
		\begin{equation}\label{eq_residual_term_convergence_occt}
			\Prob_x \Bigbraces{ \bigabsbraces{\frac{1}{n} \sum_{i=1}^{[nt]}  \indic{0<|\hfprocess{X}{}{i-1}|<\varepsilon}} > \frac{\delta}{2}} 
			\le \frac{8 \varepsilon K \sqt  }{\delta}.
		\end{equation}

		We consider the following functions:
		\begin{align}
			\phi(x)&= (2 - |x|) \indic{1\le |x|< 2} + \indic{|x|< 1}, \\
			\psi(x)&= 2(1- |x|) \indic{0.5\le |x|< 1} + \indic{|x|< 0.5}.
		\end{align}
		The functions $\phi $ and $\psi $ are both continuous and bounded with compact support and
		\begin{equation}
			\psi(x) \le \indic{|x|< 1} \le \phi(x).
		\end{equation}
		The composed function $\phi(X_t) $ and $\psi(X_t) $ are both a.s. continuous functions of $t$, hence a.s. Riemann integrable. 
		Thus,
		\begin{equation}
			\frac{1}{n} \sum_{i=1}^{[nt]} \phi(\frac{1}{\varepsilon}\hfprocess{X}{}{i-1}) \convergence{}
			\int_{0}^{t}\phi(\frac{1}{\varepsilon}X_s) \vd s,
		\end{equation}
		\begin{equation}
			\frac{1}{n} \sum_{i=1}^{[nt]} \psi(\frac{1}{\varepsilon}\hfprocess{X}{}{i-1}) \convergence{}
			\int_{0}^{t}\psi(\frac{1}{\varepsilon}X_s) \vd s .
		\end{equation}
		From Lebesgue convergence theorem both $\int_{0}^{t}\phi(\frac{1}{\varepsilon}X_s) \vd s $ and $ \int_{0}^{t}\psi(\frac{1}{\varepsilon}X_s) \vd s $ converge to $\occt{X}{0}{t} $ as $\epsilon \rightarrow 0 $.
		Thus for each $\delta >0 $ there exists an $\varepsilon_0 >0 $ such that
		for all $\varepsilon \in (0,\varepsilon_0) $,
		\begin{align}
			\Plimsup{\Prob_x}{n}	\frac{1}{n} \sum_{i=1}^{[nt]} \indic{|\hfprocess{X}{}{i-1}|< \varepsilon} &\le \occt{X}{0}{t} + \frac{\delta}{2}, \label{eq_limsup_loctim_approx} \\
			\Pliminf{\Prob_x}{n}	\frac{1}{n} \sum_{i=1}^{[nt]} \indic{|\hfprocess{X}{}{i-1}|< \varepsilon} &\ge \occt{X}{0}{t} - \frac{\delta}{2}. \label{eq_liminf_loctim_approx}
		\end{align}
		Thus, for each $\delta>0 $ there exists an $\varepsilon_0 >0 $ 
		such that
		for all $\varepsilon \in (0,\varepsilon_0) $,
		\begin{equation}\label{eq_occupation_time_convergence_band}
			\Prob_x \Bigbraces{ \bigabsbraces{\frac{1}{n} \sum_{i=1}^{[nt]}  \indic{|\hfprocess{X}{}{i-1}|<\epsilon}   - 	\occt{X}{0}{t} }  > \frac{\delta}{2}} \convergence{} 0.
		\end{equation}
		From \eqref{eq_excess_probability_decomposition}, \eqref{eq_residual_term_convergence_occt} and \eqref{eq_occupation_time_convergence_band},
		for each $\varepsilon'>0 $ by choosing $\varepsilon' = \varepsilon \delta $ in \eqref{eq_residual_term_convergence_occt} , there exists a $ \delta >0$ such that
		\begin{equation}
			0 \le	\limsup_n \Prob_x \Bigbraces{ \bigabsbraces{\frac{1}{n} \sum_{i=1}^{[nt]}  \indic{\hfprocess{X}{}{i-1}=0} - 	\occt{X}{0}{t} } > \delta} \le 4 \varepsilon K \sqt.
		\end{equation}
		Thus, $\frac{1}{n}\sum_{i=1}^{[nt]}\indic{\hfprocess{X}{}{i-1}=0} $ converges in probability to 
		$\occt{X}{0}{t}$.
		From \eqref{eq_loc_tim_part}, $\occt{X}{0}{t} $ admits almost surely a continuous version.
		Thus, from Lemma~\ref{lem_uniform_in_time_in_probability_conv}, \eqref{eq_prop_occupation_time_conv_result} is proven.
	\end{proof}

	\begin{proof}[Proof (of Corollary 		\ref{cor_stickiness_estimator_consistency})]
		From \cite[Exercise~VI.1.14]{RevYor} and the Markov property,
		$\mathcal L = \{\loct{X}{0}{t}>0\} $ almost surely.
		Thus, from \eqref{eq_auxiliary_main_result_sde} and \eqref{eq_prop_occupation_time_conv_result}, 
		on $\mathcal{L} $,
		\begin{equation}\label{eq_statistic_stickiness_pre}
			\frac{(1/n)\sum_{i=1}^{[nt]} \indic{\hfprocess{X}{}{i-1}=0}}{
				(\na/n) \sum_{i=1}^{[nt]} g_n[T](\un \hfprocesstrue{X}{i-1}) } \convergence{\Prob_x}
			\frac{\occt{X}{0}{t}}{(\lambda(g)/ \sigma(0))\loct{X}{0}{t}}.
		\end{equation}
		Since $\occt{A}{0}{t}=(\rho/2)\loct{X}{0}{t} $, from \eqref{eq_statistic_stickiness_pre}, for all $\epsilon>0 $
		\begin{align}
			\Esp_x \bigbraces{ \indic{\mathcal{L}} \indic{\absbraces{ \widehat{\rho}_{n}(X) -  \rho}> \epsilon} }
			&= 
			\Esp_x \bigbraces{ \indic{\tau_0<t} \indic{\absbraces{ \widehat{\rho}_{n}(X) -  \rho}> \epsilon} }
			\lra 0, & \text{as }\; n&\lra\infty.
		\end{align}
		From Bayes,
		\begin{align}
			\Esp_x \bigbraces{ \indic{\tau_0<t} \indic{\absbraces{ \widehat{\rho}_{n}(X) -  \rho}> \epsilon} }
			= \frac{\Prob_x \bigbraces{  \indic{\absbraces{ \widehat{\rho}_{n}(X) -  \rho}> \epsilon} \big| \tau_0<t }}{\Prob_x(\tau_0<t)} &\lra 0, & \text{as }\; n&\lra\infty.
		\end{align}
		This finishes the proof.
	\end{proof}

	\section{Numerical experiments}\label{sec_num_exp}
	In this section, we present numerical simulations to assess the asymptotic behavior of the local time approximation~\eqref{eq_auxiliary_main_result_sde} and the stickiness parameter estimator~\eqref{eq_statistic_stickiness}.
	We simulate trajectories of an approximation process of the sticky Brownian motion of parameter $\rho  $. 
	Then, we compare the stickiness parameter estimations~\eqref{eq_statistic_stickiness} with the true value of $\rho $.
	For the numerical simulations we use the Space-time Markov chain approximation or STMCA Algorithm~\cite{anagnostakis2023general}.
	This algorithm uses grid-valued random walks to approximate the law of any one-dimensional generalized diffusion process. 
	The STMCA Algorithm is particularly adapted to the problem for the following reasons:
	\begin{enumerate}
		\item It is well suited for the simulation of sticky singular one-dimensional diffusions.
		\item By suitable choice of the grid, we can control the amount of path-observations of $X$ observed through the test function $g$. 
		In particular, for grids that satisfy the condition in Corollary 2.5 of  \cite{anagnostakis2023general},
		the convergence speed of the STMCA algorithm is optimal. 
		Thus, if one increases the precision of such a grid $\bg$ around $0$, we approximate accurately $X$ with  $\bg$-valued STMCA random walks and feed more path-observations to the statistic  \eqref{eq_loct_statistic} without paying a too heavy computational cost.
	\end{enumerate}
	
	\paragraph{The statistic:} Let $X$ be a sticky Brownian motion of parameter $\rho $ and $g$ the function defined for every $x\in \IR $ by $g(x) = \indic{1<|x|<5}/8 $.
	For every $\alpha \in (0,1) $ and $n\in \IN $, we define the following path-wise statistics of $X$:
	\begin{equation}\label{eq_loct_statistic}
		\begin{aligned}
		T^{(1)}_{n,\alpha}(X) &:= \frac{n^{\alpha}}{n} \sum_{i=1}^{[nt]} g(n^{\alpha}\hfprocess{X}{\rho}{i-1} ),
		& T^{(2)}_{n,\alpha}(X) &:= \frac{2}{n}  \frac{\sum_{i=1}^{[nt]} \indic{\hfprocess{X}{}{i-1}=0}}{
			T^{(1)}_{n,\alpha}(X)}.
		\end{aligned}
	\end{equation} 
	From \eqref{eq_auxiliary_main_result_sde}, for any $\alpha \in (0,1/2)  $, the statistics $ T^{(1)}_{n,\alpha}(X) $ and $T^{(2)}_{n,\alpha}(X) $
	converge  to $\loct{X}{0}{t} $ and $\rho $ respectively in probability.
	We use  $T^{(2)}_{n,\alpha}(X) $ as proxy to assess the properties of the local time approximation.

	\paragraph{The grid:} For every $h>0 $ we define the grid:
	\begin{equation}
		\bg_1(h) = \{0\} \cup \bigcubraces{\pm x_j(h);j \in \IN},
	\end{equation}
	where  $\{x_j(h)\}_{j\in \IN} $ is defined recursively  by:
	\begin{align}
		x_0(h) &= 0, \\
		x_j(h) &= x_{j-1}(h) + \biggbraces{ \frac{h^2}{\rho} \frac{1}{x_{j-1}(h)+1} + h \Bigbraces{1 - \frac{1}{x_{j-1}(h)+1}}} \indic{x_{j-1}(h) < 1} + h \indic{x_{j-1}(h) \ge 1}.
	\end{align}
	
	\begin{figure}[h!]
		\centering
		\resizebox{0.5\textwidth}{!}{\includegraphics{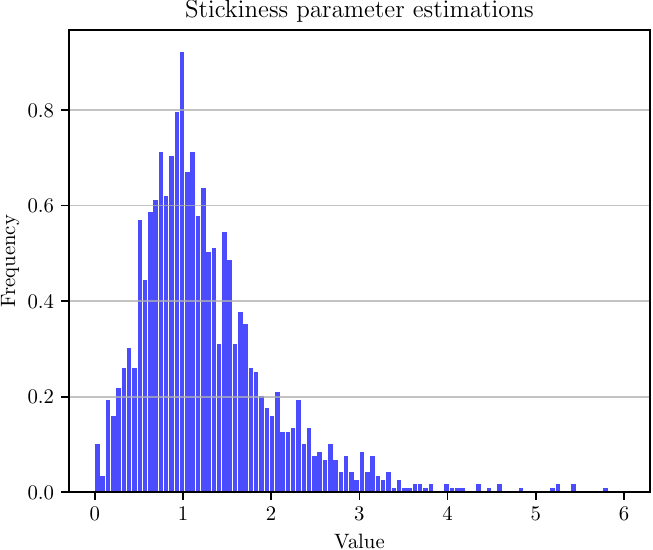}}~
		\resizebox{0.5\textwidth}{!}{\includegraphics{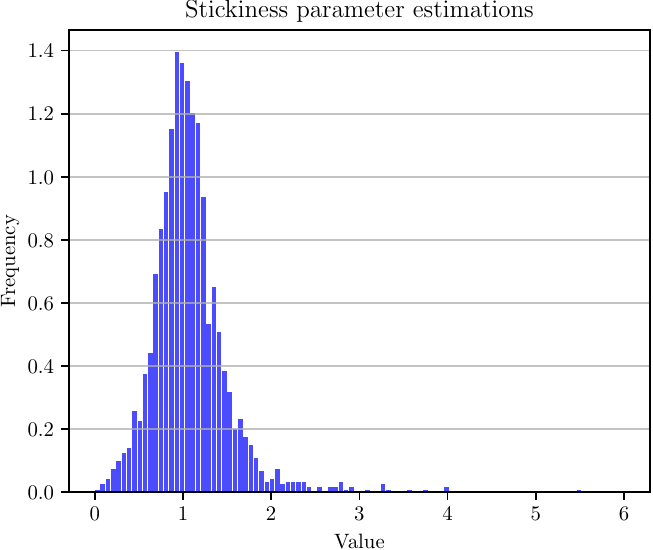}}\\
		\resizebox{0.5\textwidth}{!}{\includegraphics{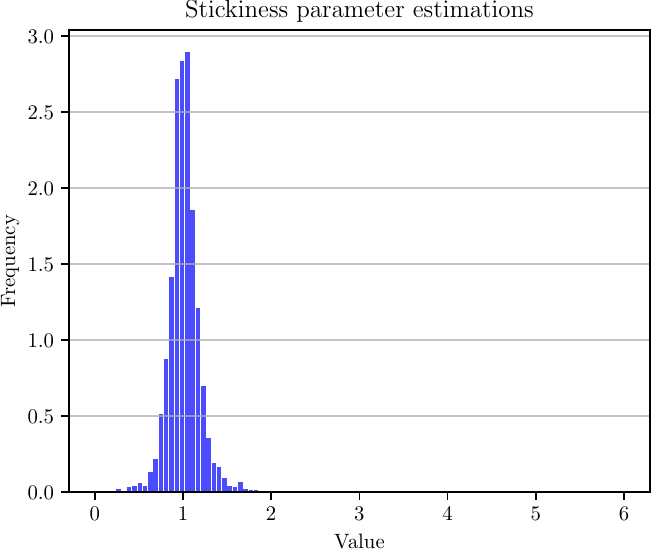}}
		\caption{ \centering
			Values of $T^{(2)}_{n,\alpha}(X) $ with $n = 100000 $ and $\alpha = 0.3$ (top-left),  $\alpha=0.45 $ (top-right),  $\alpha=0.55 $ (bottom).
			True value $\rho = 1 $.
		}
		\label{fig:estimations_histograms}
	\end{figure}

	\begin{table}[h]
		\begin{center}
			\begin{tabular}{@{}llrrrrc@{}}
				\multicolumn{1}{c}{$\alpha$} & \multicolumn{1}{c}{$n$} & \multicolumn{1}{c}{$\widehat{\rho}_{\MC}$}
				& \multicolumn{1}{c}{$\widehat{S}_{\MC}^2$} & \multicolumn{1}{c}{$\widehat{\sigma}_{\MC}$}
				& \multicolumn{1}{c}{$\widehat{\acc}$} & \multicolumn{1}{c}{$\mbox{rej}/N_{\MC}$} \\
				\hline
				0.6 & 20000 & 1.043 & 0.1451  & 0.3809 & 0.0121 & 0/1000 \\
				0.6 & 40000 & 1.029 & 0.0317 & 0.1782 & 0.0079 & 0/1000 \\
				0.6 & 100000 & 1.001 & 0.0159 & 0.1261 & 0.0046 & 0/1000 \\
				0.6 & 300000 & 1.006 & 0.0085 & 0.0922 & 0.0024 & 0/1000\\
				\hline
			\end{tabular}
		\end{center}
		\caption{Simulation results for fixed $\alpha$ and different values of $n $.}	
		
		\begin{center}
			\begin{tabular}{@{}llrrrrc@{}}
				\multicolumn{1}{c}{$\alpha$} & \multicolumn{1}{c}{$n$} & \multicolumn{1}{c}{$\widehat{\rho}_{\MC}$}
				& \multicolumn{1}{c}{$\widehat{S}_{\MC}^2$} & \multicolumn{1}{c}{$\widehat{\sigma}_{\MC}$}
				& \multicolumn{1}{c}{$\widehat{\acc}$} & \multicolumn{1}{c}{$\mbox{rej}/N_{\MC}$} \\
				\hline
				0.3 & 100000 & 1.283 & 0.667 & 0.816 & 0.12634 & 0/2000 \\
				0.4 & 100000 & 1.123 & 0.388 & 0.622 & 0.04384 & 0/2000 \\
				0.5 & 100000 & 1.040 & 0.091 & 0.302 & 0.01430 & 0/2000 \\
				0.6 & 100000 & $1.019^{*}$ & $0.049^{*}$ & $0.223^{*}$ & $0.00459^{*}$ & 1/2000 \\
				0.7 & 100000 & $1.069^{*}$ & $0.142^{*}$ & $0.378^{*}$ & $0.00136^{*}$ & 13/2000 \\
				0.8 & 100000 & $1.014^{*}$ & $0.123^{*}$ & $0.351^{*}$ & $0.00046^{*}$ & 23/2000 \\
				\hline
			\end{tabular}
		\end{center}
		\caption{Simulation results for fixed $n$ and different values of $\alpha $. Estimation with an asterisk were performed removing the trajectories with $T^{(1)}_{n}(X)=0 $}.
		
		\begin{center}
			\begin{tabular}{@{}llrrrrc@{}}
				\multicolumn{1}{c}{$\alpha$} & \multicolumn{1}{c}{$n$} & \multicolumn{1}{c}{$\widehat{\rho}_{\MC}$}
				& \multicolumn{1}{c}{$\widehat{S}_{\MC}^2$} & \multicolumn{1}{c}{$\widehat{\sigma}_{\MC}$}
				& \multicolumn{1}{c}{$\widehat{\acc}$} & \multicolumn{1}{c}{$\mbox{rej}/N_{\MC}$} \\
				\hline
				0.55536 & 20000 & 1.048 & 0.0943 & 0.3071 & 0.019029 & 0/1000\\
				0.519033 & 40000 & 1.048 & 0.0952 & 0.3086 & 0.019005 & 0/1000\\
				0.477724 & 100000 & 1.049 & 0.1008 & 0.3176 & 0.019020 & 0/1000\\
				0.436109 & 300000 & 1.049 & 0.0967 & 0.3110 & 0.019013 & 0/1000\\
				\hline
			\end{tabular}
		\end{center}
		\caption{Simulation results for $(n,\alpha) $ satisfying $\log n = 5.5/\alpha $.}\label{table_n_alpha_relation}
	\end{table}
	
	\paragraph{Simulation results:} In this section we present Monte Carlo simulation results.
	The integer $N_{\MC}$ will be the Monte Carlo simulation size.
	For every $j\le  N_{\MC}$ we simulate the path of an approximation process
	$X^{j}_t $.
	To assess the quality of each Monte Carlo estimation we use the following metrics:
	\begin{enumerate}
		\item $\bigbraces{\widehat{\rho}_{\MC}, \widehat{S}_{\MC}^2, \widehat{\sigma}_{\MC} }$: Monte Carlo estimation, variance and standard deviation of the estimated stickiness parameters,
		\item $\widehat{\acc}$: average number of path-values observed by $g$, i.e
		\begin{equation}
			\widehat{\acc} = \frac{1}{N_{\MC}} \sum_{j=1}^{N_{\MC}} \frac{1}{n}\sum_{i=1}^{n}\indic{g(X^{j}_{(i-1)/n}) \ne 0},
		\end{equation}
		\item rej: percentage of trajectories where the local time estimation equals $0$, i.e
		\begin{equation}
			\rej =	
			\# \Bigcubraces{j \le N_{\MC}:
				\frac{n^{\alpha}}{n} \sum_{i=1}^{[nt]} g(n^{\alpha}\hfprocess{X}{j}{i-1} ) = 0
			}.
		\end{equation}
	\end{enumerate}
	Within each of the following tables we use the same simulated STMCA trajectories of the sticky Brownian motion of parameter $\rho=1 $:
	
	\subsection*{Observations}
	We observe the following:
	\begin{enumerate}
		\item For fixed $n$, the higher the $\alpha$, the more the trajectory of $X^{\rho} $ is inflated and the less things are observed through $g$. 
		Is is even possible that, for some paths of $X$, the statistic $T^{(1)} $, defined in~\eqref{eq_loct_statistic}, observes nothing.
		We omit those in the Monte Carlo estimation, and use the metric $\mbox{rej} $ to keep track of their number. 
		\item Having a finite set of path-wise observations of $X^{\rho}  $, one must find an $\alpha\in (0,1) $ large enough to trigger the asymptotic regime of \eqref{eq_functional_form} and low enough so we do not dump too many trajectories.
		\item In Table \ref{table_n_alpha_relation} we see that for a fixed $c>0 $, every $(n,\alpha)$ such that $\log n = c/\alpha $ yield the same Monte Carlo variance.
		This relation can be guessed from 
		\eqref{eq_functional_form} and \eqref{eq_local_time_alternative_definition},
		\item The convergence \eqref{eq_auxiliary_main_result_sde} seems to hold for  $\alpha \in [1/2,1)$, values not covered by Theorem~\ref{thm_main_result_stqSDE_intro}. We thus conjecture the following:
	\end{enumerate}
	
	\begin{conjecture}
		Theorem~\ref{thm_main_result_stqSDE_intro} holds for any $\alpha \in (0,1) $. 
	\end{conjecture}

	\begin{small}

	\end{small}
	
\end{document}